\documentclass[11pt,a4paper]{article}
\usepackage[utf8]{inputenc}
\usepackage{amsmath}
\usepackage{amsfonts}
\usepackage{amssymb}
\usepackage{amsthm}
\usepackage{graphicx}
\usepackage{color,hyperref}
\usepackage{subfig}
\usepackage[a4paper,left=3cm,right=3cm]{geometry}
\usepackage[pagewise]{lineno}
\newtheorem{theorem}{Theorem}[section]

\newtheorem{corollary}{Corollary}[section]

\newtheorem{Algorithm}{Algorithm}[section]

\title{Reverse time migration for imaging periodic obstacles with electromagnetic plane wave}
\author{Lide Cai\thanks{\footnotesize Department of Mathematical Sciences, Tsinghua University, Beijing 100084, China. 
(cld19@mails.tsinghua.edu.cn).}
\and Junqing Chen\thanks{\footnotesize
Department of Mathematical Sciences, Tsinghua University, Beijing
100084, China. The work of this author was partially supported by National Key  R\&D Program of China 2019YFA0709600, 2019YFA0709602. (jqchen@tsinghua.edu.cn).}
}
\begin{document}
\maketitle
\begin{abstract}
  We propose reverse time migration (RTM) methods for the imaging of periodic obstacles using only measurements from lower or upper side of the obstacle arrays at a fixed frequency. We analyze the resolution of the lower side and upper side RTM methods in terms of the propagating part of the Rayleigh expansion, Helmholtz-Kirchhoff equation and the distance of the measurement surface to the obstacle arrays. We give some numerical experiments to justify the competitive efficiency of our imaging functionals and the robustness against noises.   
\end{abstract}
{\footnotesize{\bf Mathematics Subject Classification}(MSC2020): 78A46, 35R30, 65N21 }\\
{\footnotesize{\bf Keywords}:Reverse time migration, periodic structure, inverse scattering problem, resolution analysis}	

\section{Introduction}\label{sect1}
In this paper, we consider the inverse scattering problem of time-harmonic electromagnetic plane waves by periodic obstacles.
For simplicity, we assume that the periodic structure is invariant in the $x_3$ direction, and 
the scatterer is periodic in the $x_{1}$ direction, with periodicity $\Lambda$. Specifically, we have
\begin{equation}
\gamma(x_{1}+\Lambda,x_{2})=\gamma(x_{1},x_{2}).
\end{equation} 
We further assume that $\gamma(x)-1$ is  compactly supported in each $\Gamma$-periodic unit of $\mathbb{R}^{2}$ and the central unit is denoted by (see Figure \ref{fig1})
\begin{figure} 
 \centering{
 \includegraphics[width=10cm]{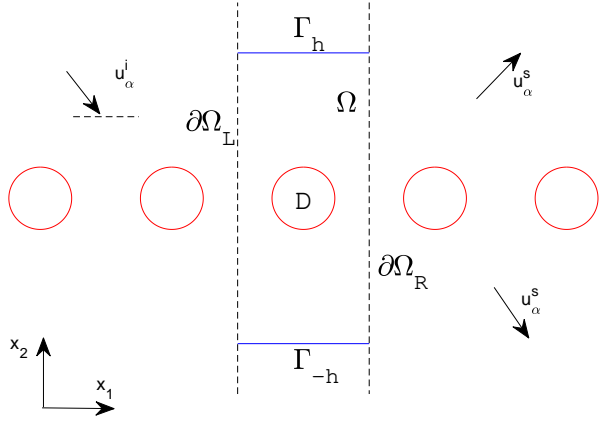}
 \caption{\label{fig1} A sample region, for periodicity $\Lambda = 2\pi$, the measurement is taken on line segment $\Gamma_{\pm h}$.}
 }
 \end{figure}
\begin{equation}
\Omega=\lbrace (x_{1},x_{2})|x_{1}\in (-\frac{\Lambda}{2},\frac{\Lambda}{2}), x_2\in \mathbb{R}\rbrace.
\end{equation}
The left and right boundaries of $\Omega$ are denoted by $\partial\Omega_{L},\partial\Omega_{R}$ respectively.  We denote the scatterer in $\Omega$ by $D$, which is the support of $\gamma(x)-1$ in $\Omega$.  

Let us consider the TE polarization in the followings. That's to say, the electric field is along the $x_3$ direction and depends on $x_1,x_2$. The incident wave is given by $u_{\alpha}^{inc}(x)=e^{i\alpha x_{1}-i\beta x_{2}}$, where $\alpha=k\cos\theta,\beta=\sqrt{k^{2}-\alpha^{2}}=k\sin\theta$, $k>0$ is the wave number and $\theta\in (0,\pi)$ is the incident angle with respect to the $x_{1}$-direction. The total field $u_\alpha$ 
is governed by the Helmholtz equation
\begin{eqnarray}
\Delta u_{\alpha} + k^{2} \gamma(x)u_{\alpha} = 0,x\in\mathbb{R}^{2}, \label{Helmholtz}
\end{eqnarray}
where $u_\alpha$ is given by
\begin{equation}
u_{\alpha}=u_{\alpha}^{s}+u^{inc}_{\alpha},x\in\mathbb{R}^{2}.\label{total}
\end{equation}
The scattered wave $u^{s}_{\alpha}$
is an $\alpha$-quasi periodic function in the $x_{1}$-direction, namely,
\begin{eqnarray*}
u^{s}_{\alpha}(x_{1}+\Lambda,x_{2})e^{-i\Lambda\alpha}=u_{\alpha}^s(x_{1},x_{2}).
\end{eqnarray*}
Furthermore,  $u^{s}_{\alpha}$ satisfies the Rayleigh expansion condition which describes the radiation of the scattered wave when $x_2\rightarrow \pm \infty $, 
\begin{eqnarray*}
u_{\alpha}^{s}(x_{1},x_{2})&=&\sum_{m\in\mathbb{Z}}u_{m,\alpha}^{s+}e^{i\alpha_{m}x_{1}+i\beta_{m}x_{2}},x_2\in \Omega_{H}^{+},\\
u_{\alpha}^{s}(x_{1},x_{2})&=&\sum_{m\in\mathbb{Z}}u_{m,\alpha}^{s-}e^{i\alpha_{m}x_{1}-i\beta_{m}x_{2}},x_2\in \Omega_{H}^{-},
\end{eqnarray*}
where, for $H>0$,   
\begin{eqnarray*}
\Omega_{H}^{+}&=&\lbrace (x_{1},x_{2})|x_{1}\in(-\frac{\Lambda}{2},\frac{\Lambda}{2}),x_{2}\geq H\rbrace,\\
\Omega_{H}^{-}&=&\lbrace (x_{1},x_{2})|x_{1}\in(-\frac{\Lambda}{2},\frac{\Lambda}{2}),x_{2}\leq -H\rbrace,
\end{eqnarray*}
which correspond to the outer diffractive regions, with 

\begin{eqnarray*}
\alpha_{m}=\alpha+\frac{2\pi}{\Lambda} m, m\in\mathbb{Z}
\end{eqnarray*}
and
\begin{eqnarray*}
\beta_{m}=\left\{\begin{array}{ll}\sqrt{k^{2}-\alpha_{m}^{2}} , &|\alpha_{m}|\leq k,\\
i\sqrt{\alpha_{m}^{2}-k^{2}}, &|\alpha_{m}|>k.
\end{array}\right.
\end{eqnarray*}
We further assume that $|\alpha_{m}|\neq k$, that is, $\alpha$ is not a wood's anomaly corresponding to $k$. As $\beta_{n}$ changes from real to imaginary as $\alpha_{n}$ passes $k$, we define index sets $B_{\alpha}$ of terms corresponding to propagating plane waves and $U_{\alpha}$ corresponding to evanescent plane waves which are
\begin{eqnarray*}
B_{\alpha}=\lbrace n\in \mathbb{Z}||\alpha_{n}|< k\rbrace \mbox{ and }
U_{\alpha}=\lbrace n\in \mathbb{Z}||\alpha_{n}|> k\rbrace.
\end{eqnarray*}  

Periodic scattering problem has long been an important topic in electromagnetic theory. It appears in extensive areas such as the optics, photonics and phononics \cite{Ammari, Bao1}.    Ever since lord Rayleigh's pioneering work in the early 20th century, a considerable amount of work has been done for the scattering problem of diffractive optics. 
Mathematically, the well-posedness of the above forward scattering problem is established, especially for the case of diffractive layers. See \cite{Kirsch1} for example. Numerically, 
in recent years, we have seen a rapid development of fast and reliable solvers. For instance, in the regime of boundary integral equation methods,  \cite{Barnett} derived a scheme stemming from free-space scattering problem with specially designed auxiliary density, while \cite{Bruno} overcome the slow convergence of the quasi-periodic Green's function \cite{Lincton1} by designing a special window function in their formulation. Furthermore, by formulating the Lippmann-Schwinger equation of quasi-periodic scattering, \cite{Lechleiter1} uses the Fourier transformation to obtain a spectral Galerkin method. Moreover, in \cite{Chen4}, an adaptive finite element PML method is developed, while in \cite{Wang} an analysis on the transparent boundary condition of the scattering problem leads to the adaptive DtN method.

Having collected several aspects of the forward periodic scattering problem, we are ready to demonstrate the inverse problem:

\textit{Given all $u^{s}_{\alpha_{n}}, n\in B_{\alpha}$ with incident waves  $u^{inc}_{\alpha_{n}}$, which are measured on $\Gamma_{h}$ or $\Gamma_{-h}$,$h \geq 0$ (see Figure \ref{fig1}), reconstruct the boundary of the support of $\gamma(x)-1$. Where $\Gamma_{\pm h}=\lbrace (x_{1},x_{2})|x_{1}\in (-\frac{\Lambda}{2},\frac{\Lambda}{2}),x_{2}=\pm h\rbrace$.}


There has been numerous literature in the inverse problem community concerning the reconstruction of periodic structure, see \cite{Ammari2, Elschner1} for the uniqueness theorems concerning the inverse problems in two and three dimensions. Further, starting from \cite{Kirsch1} etc, iterative reconstruction methods\cite{Bao8,Elschner4,Hettlich1, Elschner5,Hsiao1}, as well as other two-step reconstruction methods\cite{Elschner7} are designed and studied. Especially, see \cite{Bao1} for a comprehensive survey on the reconstruction methods for periodic grating profiles.
As for the direct imaging methods on the periodic structures, 
there are a number of studies on diffractive periodic structures, such as factorization method \cite{Arens1}, the  linear sampling methods \cite{Yang1, Yang2}. In addition, for the reconstruction of periodically compactly-supported obstacles, there has recently been literature on the design of special indicator functional\cite{Nyugen2}, where a comparison of the above direct imaging methods in the case of periodic scattering is included.

 As is known, the RTM method has a competitive resolution of the bounded obstacle if one obtains the full-aperture data \cite{Chen1, Chen2, Chen3}. Extending the RTM method to the case of unbounded surface scattering, \cite{Li} is able to reconstruct simultaneously the locally-perturbed half-space and a compactly supported obstacle.  On the other hand, using limited aperture data, an analysis of half-space RTM in \cite{Chen1} indicates that one can obtain partially the boundary of the obstacle, whose resolution is given by the Kirchhoff coefficient, which is closely connected with the opening of the scatterer. 

The major contribution of this paper is an investigation of RTM method in inverse periodic scattering problem.  We analyze the RTM method with quasi-periodic data of measurements only from below or above the periodic array.  The resolution analysis is based on the Helmholtz-Kirchhoff equation, Rayleigh expansion of the scattering wave and point spread function for quasi-periodic scattering problem. Specially, we prove that the lower side RTM imaging functional is positive, the lower side and upper side RTM functional peak at the boundary of the scatterer.  We also demonstrate numerically that with partial data only from below the periodic array, one can find the clear image of the vertical part of the periodic array other than its lower part, and the imaging functional has the nice property of positivity. While with partial data only from above the periodic array, one can obtain the clear image of the horizontal part of the periodic array.

The structure of the paper is as follows:
To begin with, in Section \ref{sect2}, we investigate several preliminary tools for the resolution analysis of RTM functional. By an investigation of the Helmholtz-Kirchhoff equation for the quasi-periodic scattering problem, we observe a natural point spread function corresponding to the propagating modes of the quasi-periodic Green's function in terms of its spectral decomposition. 
In Section \ref{sect3}, where the RTM algorithm is proposed, the special structure of the point spread function leads to the form of cross-correlation between the incident waves of propagating modes and the back-propagation of the received data. 
Section \ref{sect4} presents the resolution analysis regarding the imaging power of RTM functional. We start with the lower RTM method which is constructed by measurements from below the obstacle, while the second part of the section gives a further analysis on the resolution of the upper RTM method. We extend our analysis to the sound-soft obstacle for the RTM methods in Section \ref{sect5}.
Finally in Section \ref{sect6}, the numerical experiments thus demonstrate the competitive imaging ability of our RTM functionals.
\section{Preliminaries}\label{sect2}
\subsection{Point spread function}
We begin by recalling the quasi-periodic Green's function \cite{Lincton1} in $\mathbb{R}^{2}$.
\begin{eqnarray}
G^{qp}_{\alpha}(x,y)=\frac{i}{4}\sum_{n\in \mathbb{Z}}H_{0}^{1}(k|x-y_{n}|)e^{in\Lambda\alpha}.\label{qpgreen}
\end{eqnarray}
Here, $y_{n}=(y_{n}^{1},y_{n}^{2})=(y_{1}+n\Lambda,y_{2})$ and $H^1_0(\cdot)$ is the zeroth-order Hankel function of the first kind. It is the solution to the equation
\begin{eqnarray*}
\Delta G^{qp}_{\alpha}(x,y) + k^{2}G^{qp}_{\alpha}(x,y)=-\sum_{n\in \mathbb{Z}}\delta_{y_{n}^{1}}(x_{1})\delta_{y_{n}^{2}}(x_{2})e^{in\Lambda\alpha}.
\end{eqnarray*}
The physical interpretation of \eqref{qpgreen} is the wave emitted by a periodic array of point sources each of which is equipped with a phase shift in the $x_{1}$ direction. 

Using the Poisson summation formula, we may obtain the spectral representation of the quasi-periodic Green's function at non-wood's anomalies, see \cite{Bao1}
\begin{equation}
G^{qp}_{\alpha}(x,y)=\frac{i}{2\Lambda}\sum_{n\in\mathbb{Z}}\frac{1}{\beta_{n}}e^{i\alpha_{n}(x_{1}-y_{1})+i\beta_{n}|x_{2}-y_{2}|}.\label{spgreen}
\end{equation}

Now we can give the Helmholtz-Kirchhoff's equation for quasi-periodic Green's function. To start with, we consider the case where the source points of the Green's function are above the upper measurement surface.
\begin{theorem}\label{thm:2.1}
For $y,z\in\Omega_{-h}^{+}=\{x=(x_1,x_2)|x\in\Omega,x_2> -h,h > 0\}$, assume that $\alpha$ is not a wood's anomaly, we have the following Helmholtz-Kirchhoff's equation
\begin{eqnarray}
\int_{\Gamma_{-h}}\frac{\partial\overline{G^{qp}_{\alpha}(x,y)}}{\partial x_{2}}G^{qp}_{\alpha}(x,z)-\frac{\partial G^{qp}_{\alpha}(x,z)}{\partial x_{2}}\overline{G^{qp}_{\alpha}(x,y)}ds(x)=F^{L}_{\alpha}(y,z).\label{HK}
\end{eqnarray}
Here, $\Gamma_{-h}=\{(x_1,x_2)|x_2 = -h, -\Lambda/2< x_1< \Lambda/2\}$ and
\begin{equation}
F^{L}_{\alpha}(y,z)=\frac{i}{2\Lambda}\sum_{n\in B_{\alpha}}\frac{1}{\beta_{n}}e^{i\alpha_{n}(y_{1}-z_{1})-i\beta_{n}(y_{2}-z_{2})}.\label{pspread}
\end{equation}
\end{theorem}
\begin{proof}
 To start with, we observe that, for $y_{2},z_{2}\geq-h$, $|-h-y_{2}|=(y_{2}+h)$, $|-h-z_{2}|=(h+z_{2})$. Further 
 \begin{eqnarray}
      \overline{i\beta_{n}}=
      \left\{
      \begin{aligned}
      &-i\beta_{n},\quad n\in B_{\alpha}\\
      &i\beta_{n},\quad n\in U_{\alpha}
      \end{aligned}.
 \right.
 \end{eqnarray}
 Thus the absolute value in the spectral form \eqref{spgreen} of the quasi-periodic Green's function can be exactly calculated.  With the help of the orthogonality 
\begin{eqnarray*}
    \int_{\Gamma_-h} e^{i\alpha_nx_1}e^{-i\alpha_mx_1}ds(x) =\left\{
      \begin{aligned}
      &0,\quad n\neq m\\
      &\Lambda,\quad n=m
      \end{aligned}
 \right. ,
\end{eqnarray*}
we obtain that
\begin{eqnarray}
&&\int_{\Gamma_{-h}}\frac{\partial\overline{G^{qp}_{\alpha}(x,y)}}{\partial x_{2}}G^{qp}_{\alpha}(x,z)ds(x)\nonumber\\
&&\quad=\frac{1}{4\Lambda^{2}}\int_{\Gamma_{-h}}\sum_{n\in\mathbb{Z}}i\overline{e^{i\alpha_{n}(x_{1}-y_{1})+i\beta_{n}(y_{2}+h)}}\cdot\sum_{n\in\mathbb{Z}}\frac{1}{\beta_{n}}e^{i\alpha_{n}(x_{1}-z_{1})+i\beta_{n}(z_{2}+h)}ds(x)\nonumber\\
&&\quad=\frac{1}{4\Lambda}(\sum_{n\in B_{\alpha}}\frac{i}{\beta_{n}}e^{i\alpha_{n}(y_{1}-z_{1})-i\beta_{n}(y_{2}-z_{2})}+\sum_{n\in U_{\alpha}}\frac{i}{\beta_{n}}e^{i\alpha_{n}(y_{1}-z_{1})+i\beta_{n}(2h+(y_{2}+z_{2}))}),\label{eva1}
\end{eqnarray}
and
\begin{eqnarray*}
&&\int_{\Gamma_{-h}}\frac{\partial{G^{qp}_{\alpha}(x,z)}}{\partial x_{2}}\overline{G^{qp}_{\alpha}(x,y)}ds(x)\\
&&\qquad=\frac{1}{4\Lambda^{2}}\int_{\Gamma_{-h}}\sum_{n\in\mathbb{Z}}-i{e^{i\alpha_{n}(x_{1}-z_{1})+i\beta_{n}(z_{2}+h)}}\cdot\sum_{n\in\mathbb{Z}}\overline{\frac{1}{\beta_{n}}e^{i\alpha_{n}(x_{1}-y_{1})+i\beta_{n}(h+y_{2})}}ds(x)\\
&&\qquad=\frac{1}{4\Lambda}(\sum_{n\in B_{\alpha}}\frac{-i}{\beta_{n}}e^{i\alpha_{n}(y_{1}-z_{1})-i\beta_{n}(y_{2}-z_{2})}+\sum_{n\in U_{\alpha}}\frac{i}{\beta_{n}}e^{i\alpha_{n}(y_{1}-z_{1})+i\beta_{n}(2h+(y_{2}+z_{2}))}).
\end{eqnarray*}
Thus doing the subtraction, we obtain \eqref{HK} and complete the proof.
\end{proof}
Inspecting the proof of Theorem \ref{thm:2.1}, we obtain that, 
\begin{corollary}\label{cor:2.1}
For $y,z\in\Omega^+_{-h}=\{x=(x_1,x_2)|x\in\Omega, x_2>-h, h> 0\}$,  assume that $\alpha$ is not a wood's anomaly, we have the following asymptotic result
\begin{eqnarray}
\int_{\Gamma_{-h}}\frac{\partial\overline{G^{qp}_{\alpha}(x,y)}}{\partial x_{2}}G^{qp}_{\alpha}(x,z)d s(x)=\frac{1}{2}F^{L}_{\alpha}(y,z)+R^L_{\alpha}(y,z;h),\label{HK1}
\end{eqnarray}
where $|R^L_{\alpha}(y,z;h)|=O(h^{-1}),|\nabla R^L_{\alpha}(y,z;h)|=O(h^{-1})$as $h\rightarrow\infty$.
\end{corollary}
\begin{proof}
Denote by 
\begin{eqnarray*}
    R^L_{\alpha}(y,z;h)=\frac{1}{4\Lambda}\sum_{n\in U_{\alpha}}\frac{i}{\beta_{n}}e^{i\alpha_{n}(y_{1}-z_{1})-i\beta_{n}(-2h-(y_{2}+z_{2}))}, h>0\label{resL}
\end{eqnarray*} 
and 
\begin{eqnarray*}
    \beta_{\Delta_{\alpha}}=\min\{|i\beta_{n}|| n\in U_{\alpha}\},
\end{eqnarray*}
from \eqref{eva1}, 
we obtain that
\begin{eqnarray*}
&&4\Lambda|(R^L_{\alpha}(y,z;h))|=|\sum_{n\in U_{\alpha}}\frac{i}{\beta_{n}}e^{i\alpha_{n}(y_{1}-z_{1})-i\beta_{n}(-2h-(y_{2}+z_{2}))}|\\
&&\quad=|\sum_{n\in U_{\alpha}}\frac{1}{\sqrt{\alpha_{n}^{2}-k^{2}}}e^{-\sqrt{\alpha_{n}^{2}-k^{2}}(2h+(y_{2}+z_{2}))}e^{i\alpha_{n}(y_{1}-z_{1})}| \\
&&\quad\leq 2(\frac{1}{\beta_{\Delta_{\alpha}}}e^{-\beta_{\Delta_{\alpha}}(2h+z_{2}+y_{2})}+\int_{k}^{\infty}\frac{1}{\sqrt{s^{2}-k^{2}}}e^{-\sqrt{s^{2}-k^{2}}(2h+(y_{2}+z_{2}))}ds)\\
&&\quad\leq 2(\frac{1}{\beta_{\Delta_{\alpha}}}e^{-\beta_{\Delta_{\alpha}}(2h+z_{2}+y_{2})}+\int_{0}^{\infty}\frac{1}{t}e^{-t(2h+(y_{2}+z_{2}))}\frac{t}{\sqrt{t^{2}+k^{2}}}dt)\\
&&\quad\leq 2(\frac{1}{\beta_{\Delta_{\alpha}}}e^{-\beta_{\Delta_{\alpha}}(2h+z_{2}+y_{2})}+\int_{0}^{\infty}\frac{1}{k}e^{-t(2h+(y_{2}+z_{2}))}dt)\\
&&\quad\leq 2(\frac{1}{\beta_{\Delta_{\alpha}}}e^{-\beta_{\Delta_{\alpha}}(2h+z_{2}+y_{2})}+\frac{1}{k(2h+y_{2}+z_{2})}).
\end{eqnarray*}
Similarly, we may prove the estimate for $|\nabla R_{\alpha}^{L}(y,z)|$, and
thus completes the proof of the lemma.
\end{proof}

Thus, it's reasonable to hope that the cross correlation between the quasi-periodic equation and the conjugate of its $x_{2}$ derivative would produce a good approximation of $F^{L}_{\alpha}(y,z)$, whose imaginary part of has the form of a point spread function that peak at $y=z$, and decay as $y$ leaves $z$. In Section \ref{sect4}, we shall see that this quasi-periodic point spread function reflect the imaging ability of the lower RTM functional $\mathcal{I}_{L}(z)$

Following similar proof as Theorem \ref{thm:2.1}, we obtain that the cross-correlation between quasi-periodic Green's function for $y,z\in \Omega_{h}^{-}$.
\begin{theorem}\label{thm:2.2}
For $y,z\in\Omega_{h}^{-}=\{x=(x_1,x_2)|x\in\Omega,x_2 < h,h > 0\}$, assume that $\alpha$ is not a wood's anomaly, we have the following Helmholtz-Kirchhoff's equation
\begin{eqnarray*}
\int_{\Gamma_{h}}\frac{\partial\overline{G^{qp}_{\alpha}(x,y)}}{\partial x_{2}}G^{qp}_{\alpha}(x,z)-\frac{\partial G^{qp}_{\alpha}(x,z)}{\partial x_{2}}\overline{G^{qp}_{\alpha}(x,y)}ds(x)=-F^{U}_{\alpha}(y,z).
\end{eqnarray*}
Here, $\Gamma_{h}=\{(x_1,x_2)|x_2 = h, -\Lambda/2< x_1< \Lambda/2\}$ and
\begin{eqnarray*}
F^{U}_{\alpha}(y,z)=\frac{i}{2\Lambda}\sum_{n\in B_{\alpha}}\frac{1}{\beta_{n}}e^{i\alpha_{n}(y_{1}-z_{1})+i\beta_{n}(y_{2}-z_{2})}.\label{pspread1}
\end{eqnarray*}
\end{theorem}
\begin{corollary}
For $y,z\in\Omega^-_{h}=\{x=(x_1,x_2)|x\in\Omega, x_2<h, h>0\}$, assume that $\alpha$ is not a wood's anomaly, we have the following asymptotic result
\begin{equation}
\int_{\Gamma_{h}}\frac{\partial\overline{G^{qp}_{\alpha}(x,y)}}{\partial x_{2}}G^{qp}_{\alpha}(x,z)d s(x)=-\frac{1}{2}F^{U}_{\alpha}(y,z)+R^U_{\alpha}(y,z;h),\label{HK2}
\end{equation}
where $|R^U_{\alpha}(y,z;h)|=O(h^{-1}),|\nabla R^U_{\alpha}(y,z;h)|=O(h^{-1})$as $h\rightarrow\infty$. 
\end{corollary}

The difference of $+$,$-$ in \eqref{pspread1} and \eqref{pspread}, leads to a sharper point spread function for RTM functionals. Namely, if we consider the following representations, 
\begin{eqnarray}
&&F^{1}_{\alpha}(y,z)=\frac{i}{2\Lambda}\sum_{n\in B_{\alpha}}\frac{1}{\beta_{n}}e^{i\alpha_{n}(y_{1}-z_{1})}\cos{\beta_{n}(y_{2}-z_{2})},\label{Falpha1}\\
&&F^{2}_{\alpha}(y,z)=\frac{i}{2\Lambda}\sum_{n\in B_{\alpha}}\frac{1}{\beta_{n}}e^{i\alpha_{n}(y_{1}-z_{1})}\sin{\beta_{n}(y_{2}-z_{2})},\label{Falpha2}
\end{eqnarray}
we have 
\begin{eqnarray*}
F^{U}_{\alpha}(y,z)=F^{1}_{\alpha}(y,z)+iF^{2}_{\alpha}(y,z),\quad F^{L}_{\alpha}(y,z)=F^{1}_{\alpha}(y,z)-iF^{2}_{\alpha}(y,z).
\end{eqnarray*}
Now, using the spectral expansion of the quasi-periodic Green's function, we obtain the next theorem.
\begin{theorem}\label{thm:2.3}
Assume that $|\alpha_{n}|\neq k, n\in\mathbb{Z}$,$|y-z|\neq m\Lambda,m\in\mathbb{Z}$, we have the following result
\begin{equation}
F^{1}_{\alpha}(y,z)=\sum_{n\in\mathbb{Z}}\frac{i}{4}J^{1}_{0}(k|y-z_{n}|)e^{in\Lambda\alpha},\label{Falpha12}
\end{equation}
where $J^{1}_{0}(k|x-y|)$ is the Bessel function of the first kind.
\end{theorem}

\begin{proof}
Being aware of that 
\begin{eqnarray*}
G^{qp}_\alpha(z,y)&=&\sum_{n\in\mathbb{Z}}\frac{i}{4}H^{1}_{0}(k|z-y_{n}|)e^{in\alpha}
=\sum_{n\in\mathbb{Z}}\frac{i}{4}H^{1}_{0}(k|y-z_{-n}|)e^{in\Lambda\alpha}\\
&=&\sum_{n\in\mathbb{Z}}\frac{i}{4}H^{1}_{0}(k|y-z_{n}|)e^{-in\Lambda\alpha},
\end{eqnarray*}
Since $\alpha_n\neq k$ and $|y-z|\neq n\Lambda, \forall n\in\mathbb{Z}$, the series is convergent. Then we have
\begin{eqnarray*}
    G^{qp}_{\alpha}(y,z)-\overline{G^{qp}_{\alpha}(z,y)}&=&\sum_{n\in\mathbb{Z}}\frac{i}{4}H^{1}_{0}(k|y-z_{n}|)e^{in\alpha}+\sum_{n\in\mathbb{Z}}\frac{i}{4}\overline{H^{1}_{0}(k|z-y_{n}|)}e^{in\Lambda\alpha}\\
    &=&\frac{i}{4}\sum_{n\in\mathbb{Z}}(H^{1}_{0}(k|y-z_{n}|)+\overline{H^{1}_{0}(k|y-z_{n}|)})e^{in\Lambda\alpha}\\
    &=&\sum_{n\in\mathbb{Z}}\frac{i}{2}J_{0}(k|x-y_{n}|)e^{in\Lambda\alpha}.
\end{eqnarray*}
On the other hand
\begin{eqnarray*}
     G^{qp}_{\alpha}(y,z)-\overline{G^{qp}_{\alpha}(z,y)}&=&\frac{i}{2\Lambda}\sum_{n\in \mathbb{Z}}\frac{1}{\beta_{n}}e^{i\alpha_{n}(y_{1}-z_{1}
     +i\beta_{n}|y_{2}-z_{2}|}
     +\frac{1}{\bar{\beta_{n}}}e^{i\alpha_{n}(y_{1}-z_{1})-i\overline{\beta_{n}}|y_{2}-z_{2}|}\\
     &=&\frac{i}{\Lambda}\sum_{n\in B_{\alpha}}\frac{1}{\beta_{n}}e^{i\alpha_{n}(y_{1}-z_{1})}\cos(\beta_{n}(y_{2}-z_{2}))\\
     &=&2F_{\alpha}^{1}(y,z)
\end{eqnarray*}
 which is exactly what we have asserted.
\end{proof}
We remark that right hand side of \eqref{Falpha12} also appeared in \cite{Nyugen2} as the point spread function. \eqref{Falpha12} resembles that $F_{\alpha}^{1}$ has a similar imaging power as the Bessel's function of the first kind. We further include Figure \ref{figp} of the point spread functions $F^{L}_{\alpha}$ and $F^{1}_{\alpha}$. From Figure \ref{figp:a} and \ref{figp:b} which correspond to $Im(F^{L}_{\alpha})$, and Figure \ref{figp:c} and \ref{figp:d} which correspond to $Im(F_{\alpha}^1)$, it's clear that the point spread function $F_{\alpha}^{1}$ has the similar behavior with $F_{\alpha}$.
\begin{figure}
\centering
\subfloat[\label{figp:a}]{\includegraphics[width=6cm]{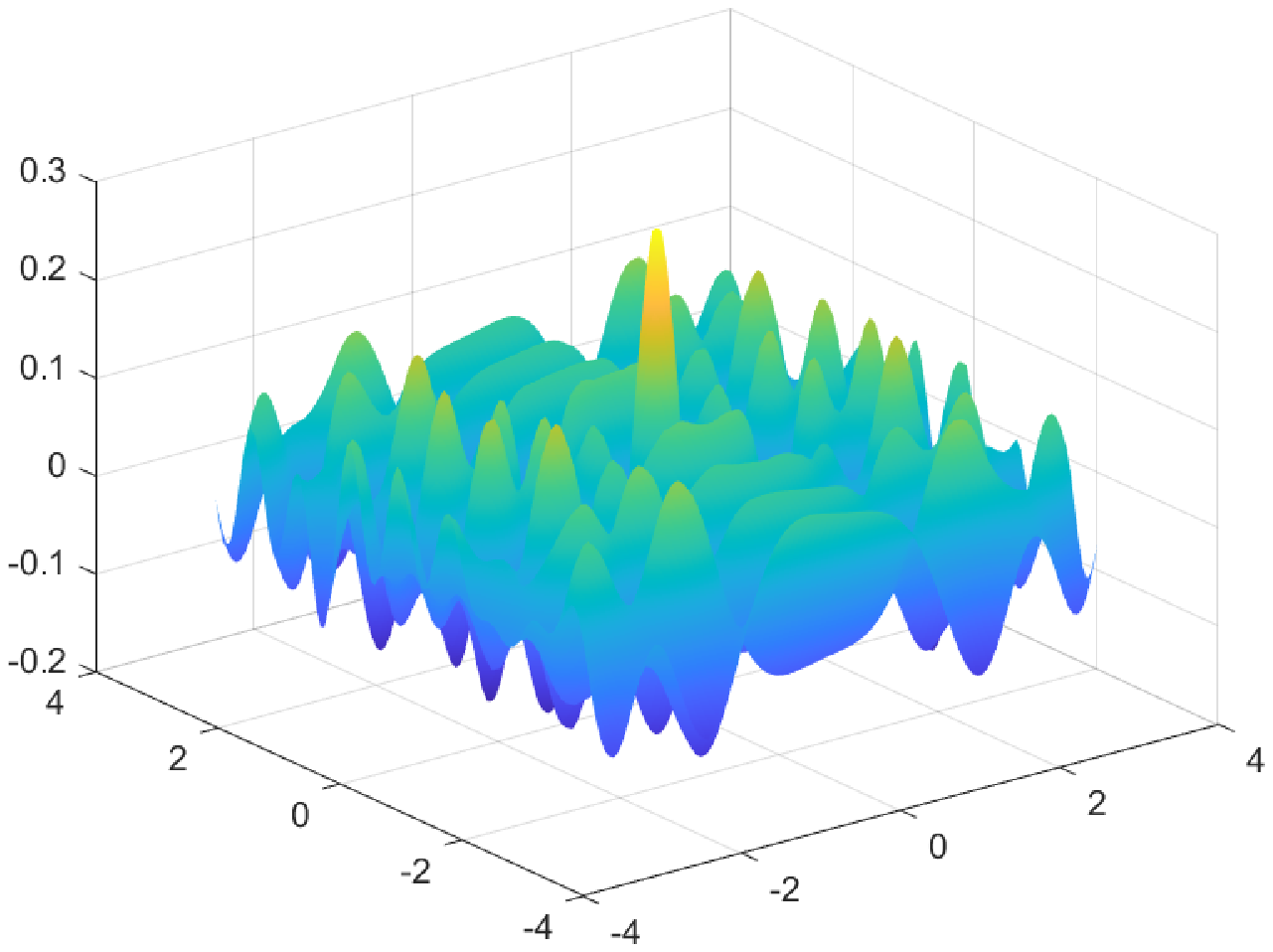}}\quad
\subfloat[\label{figp:b}]{\includegraphics[width=6cm]{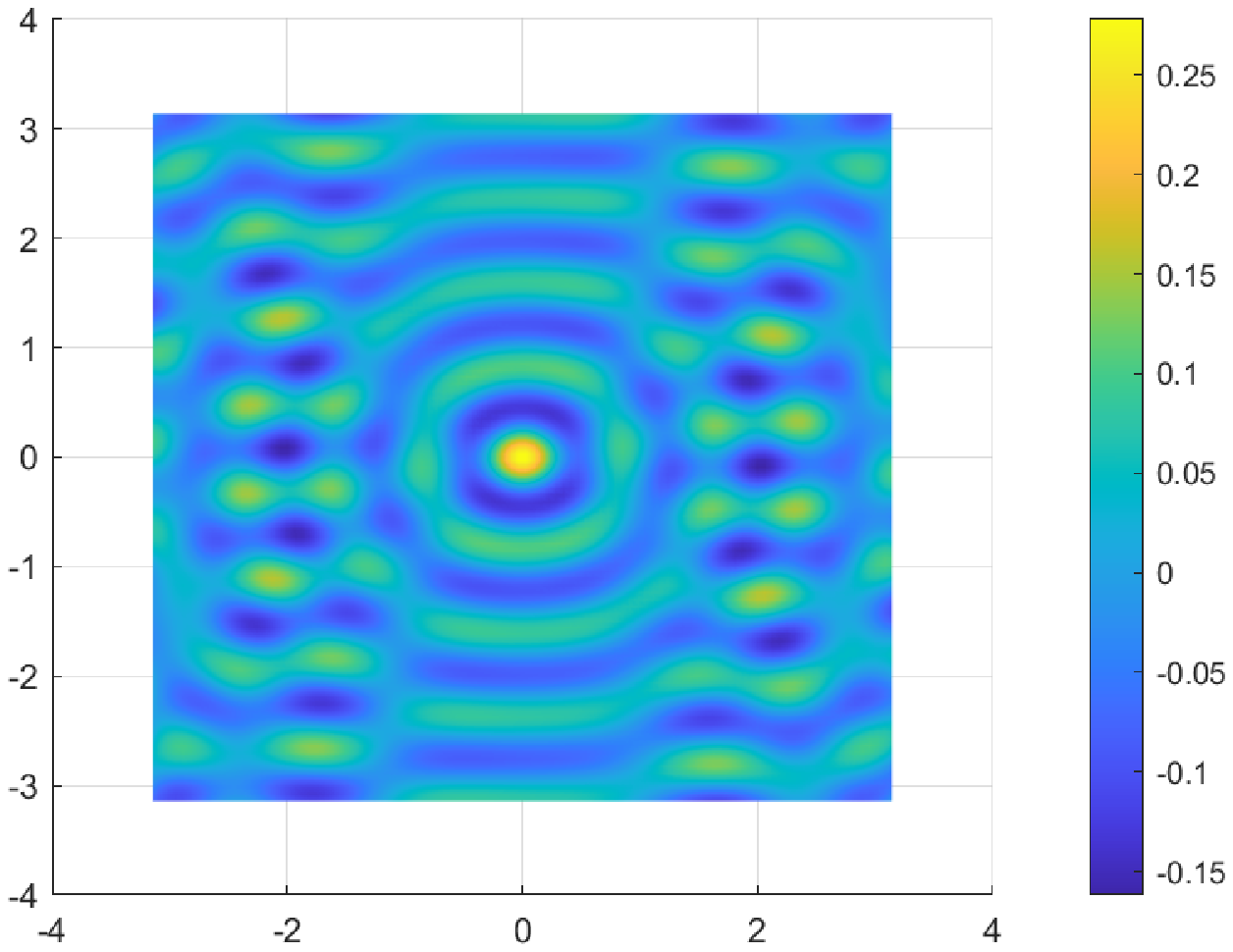}}\quad
\centering
\subfloat[\label{figp:c}]{\includegraphics[width=6cm]{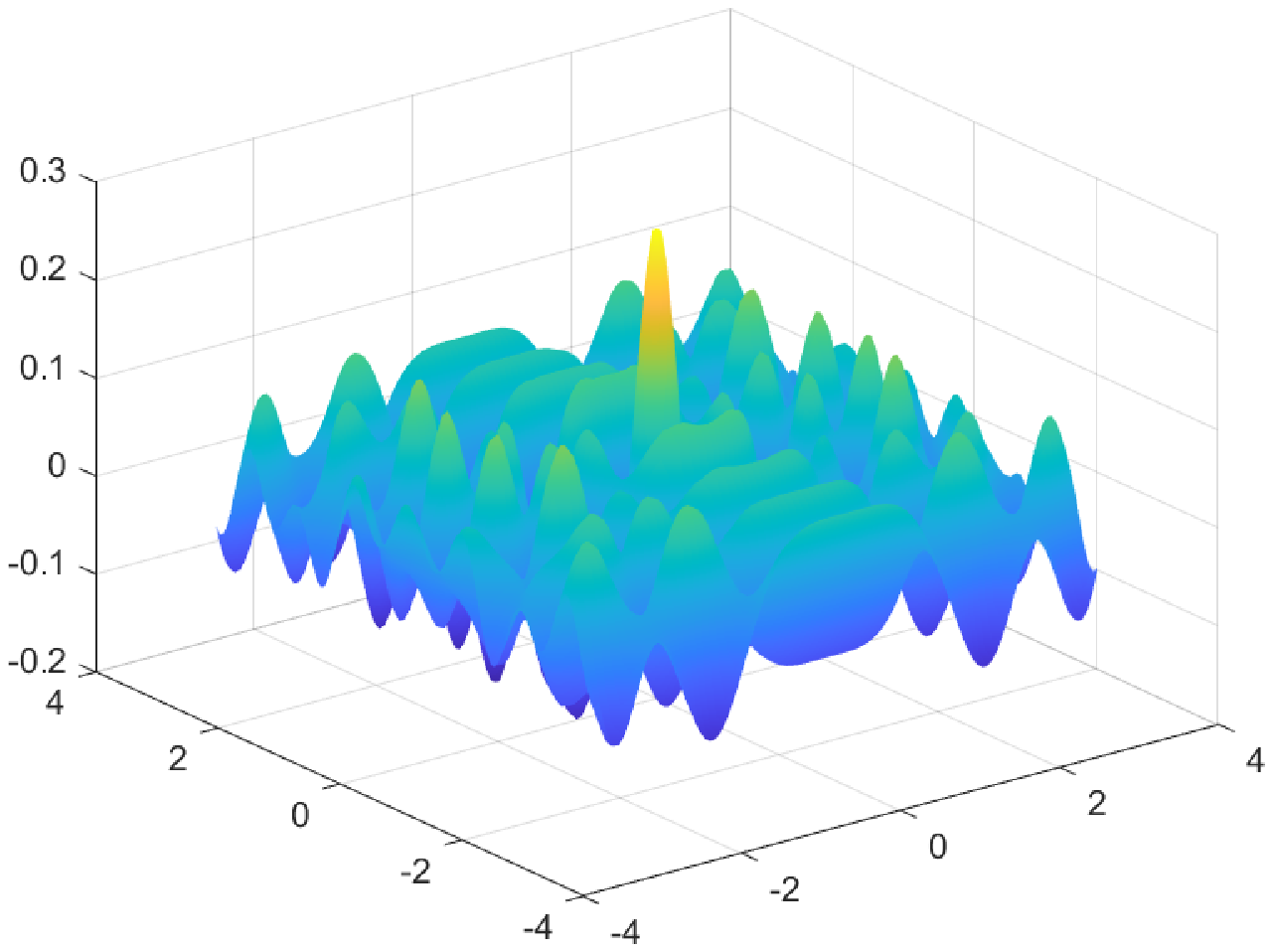}}\quad
\subfloat[\label{figp:d}]{\includegraphics[width=6cm]{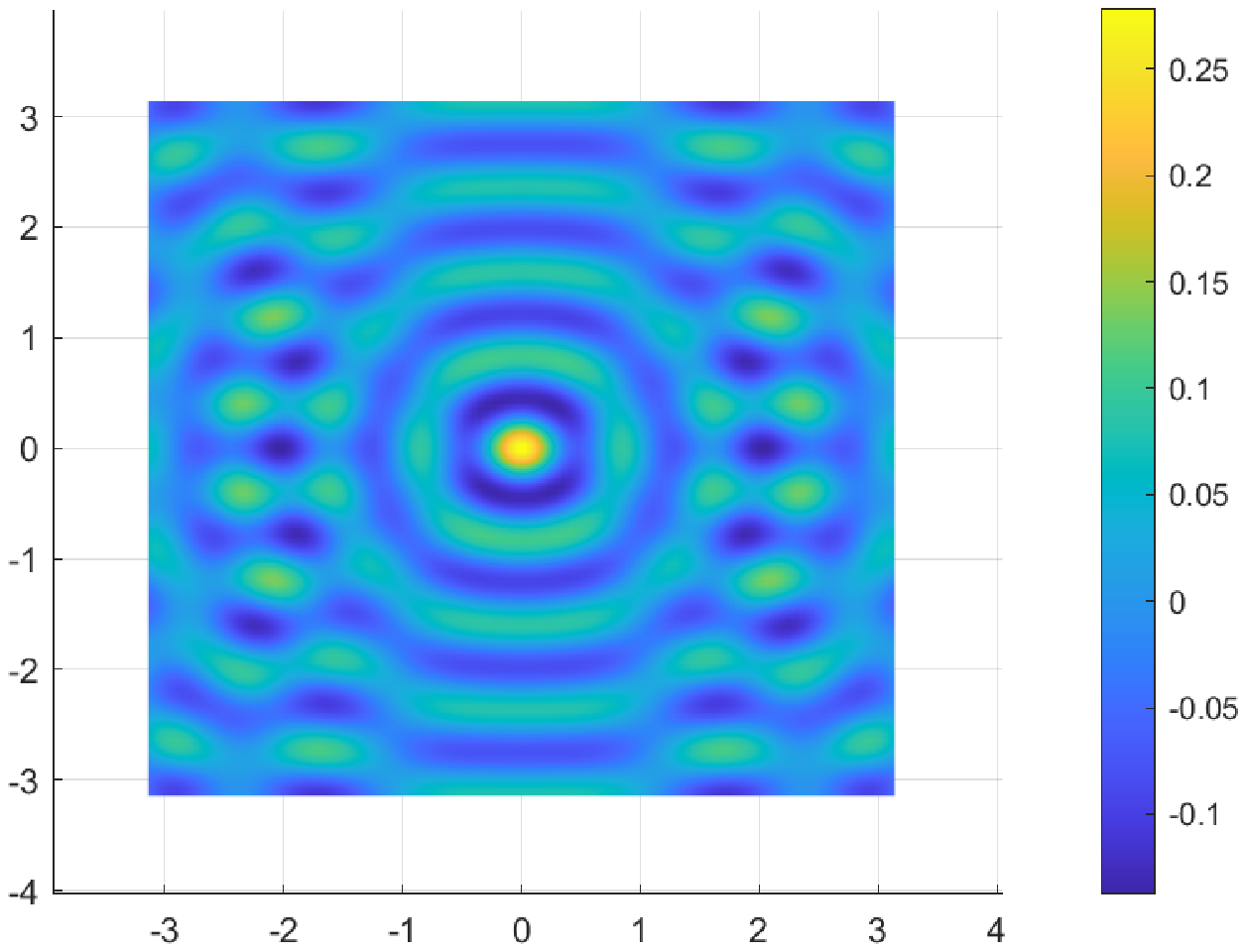}}
\caption{The imaginary part of point-spread function $F^{L}_{\alpha}$ and $F_{\alpha}^{1}$, for $k=5.2\pi,x\in\Omega,\Lambda=2\pi$.}
\label{figp}
\end{figure}
This similarity is reflected both theoretically and numerically in Section \ref{sect4} and \ref{sect5}.

The second basic ingredient is the following Lippmann-Schwinger equation
\begin{theorem}\label{thm:2.4}
Denote by $D$ the compact support of $\gamma(x)-1$ in the periodic cell $\Omega$, the $\alpha$-quasi-periodic solution to the Helmholtz equation satisfies the following quasi-periodic Lippmann-Schwinger equation. 
\begin{eqnarray}
u_{\alpha}^{s}(x)=\int_{D}k^{2}(\gamma(y)-1)G^{qp}_{\alpha}(x,y)u_{\alpha}(y)dy.\label{qpLS}
\end{eqnarray} 
for all $x\in\Omega\setminus\bar{D}$.
\begin{proof}
With \eqref{Helmholtz}, \eqref{total}, we have
\begin{eqnarray*}
\Delta u^{s}_{\alpha}(y)+k^{2}u^{s}_{\alpha}(y) = k^2(1-\gamma(y))u_{\alpha}(y).
\end{eqnarray*}
For any $x\in\Omega$ outside of $D$, taking a small ball $B_{\rho}(x)$ of radius $\rho$ that is contained in $\Omega$. Further, bound the outside of $D$ with a rectangular section $\partial\Omega_{0}$, whose upper and lower bounds are given by $\Gamma_{\underline{+}h}$ with $h\geq H$. The region include in the section is denoted by $\Omega_{0}$. Thus 
multiplying both sides of the equations by $G^{qp}_{\alpha}(x,y)$, and integrate over the $\Omega_{0}\setminus B_{\rho}(x)$, we have
\begin{eqnarray*}
\int_{\Omega_{0}\setminus B_{\rho}(x)}G^{qp}_{\alpha}(x,y)(\Delta u^{s}_{\alpha}(y)+k^{2}u^{s}_{\alpha}(y))dy=\int_{D}k^2(1-\gamma(y))G^{qp}_{\alpha}(x,y)u_{\alpha}(y)dy.
\end{eqnarray*}
Using the Green's second formula, we obtain that 
\begin{eqnarray}
LHS=\int_{\partial(\Omega_{0}\setminus B_{\rho}(x))}G^{qp}_{\alpha}(x,y)\frac{\partial u^{s}_{\alpha}(y)}{\partial\nu(y)}-u^{s}_{\alpha}(y)\frac{\partial G^{qp}_{\alpha}(x,y)}{\partial\nu(y)}ds(y).\label{int}
\end{eqnarray}
We denote by $RHS=I_{1}+I_{2}$ the integration on the two boundaries in \eqref{int} respectively, thus on $\partial \Omega_{0}$, 
\begin{eqnarray}
I_{1}=(\int_{\Gamma_{\underline{+}h}}+\int_{\partial\Omega_{0}^{L}}-\int_{\partial\Omega_{0}^{R}})G^{qp}_{\alpha}(x,y)\frac{\partial u_{\alpha}^{s}(y)}{\partial\nu(y)}-u^{s}_{\alpha}(y)\frac{\partial G^{qp}_{\alpha}(x,y)}{\partial\nu(y)}ds(y).\label{intI1}
\end{eqnarray}
Observing that $G^{qp}_{\alpha}(x,y)$ is $-\alpha$-quasi-periodic in $y_{1}$, we have that the difference of integration along the left and right boundaries in \eqref{intI1} vanishes, and we are left with
\begin{eqnarray*}
I_{1}=\int_{\Gamma_{\underline{+}h}}G^{qp}_{\alpha}(x,y)\frac{\partial u_{\alpha}^{s}(y)}{\partial y_{2}}-u^{s}_{\alpha}(y)\frac{\partial G^{qp}_{\alpha}(x,y)}{\partial y_{2}}ds(y).
\end{eqnarray*}
Using the Spectral representation of $G^{qp}_{\alpha}(x,y)$ and the Rayleigh expansion of $u_{\alpha}^{s}$ on  $\Gamma_{\underline{+}h}$, since that $x\in \Omega_{0},x_{2}\in(-h,h),h\geq H$, we have
\begin{eqnarray*}
    &&\int_{\Gamma_{h}}G^{qp}_{\alpha}(x,y)\frac{\partial u_{\alpha}^{s}(y)}{\partial y_{2}}-u^{s}_{\alpha}(y)\frac{\partial G^{qp}_{\alpha}(x,y)}{\partial y_{2}}ds(y)\\
    &&=\frac{i}{2\Lambda}\int_{\Gamma_{h}}(\sum_{n\in\mathbb{Z}}\frac{1}{\beta_{n}}e^{i\alpha_{n}(x_{1}-y_{1})+i\beta_{n}(h-x_{2})}\cdot\sum_{m\in\mathbb{Z}}i\beta_{m}u_{m,\alpha}^{s+}e^{i\alpha_{m}y_{1}+i\beta_{m}h}\\
    &&-\sum_{n\in\mathbb{Z}}ie^{i\alpha_{n}(x_{1}-y_{1})+i\beta_{n}(h-x_{2})}\cdot\sum_{m\in\mathbb{Z}}u_{m,\alpha}^{s+}e^{i\alpha_{m}y_{1}+i\beta_{m}h})ds(y)\\
    &&=0.
\end{eqnarray*}
and
\begin{eqnarray*}
    &&\int_{\Gamma_{-h}}G^{qp}_{\alpha}(x,y)\frac{\partial u_{\alpha}^{s}(y)}{\partial y_{2}}-u^{s}_{\alpha}(y)\frac{\partial G^{qp}_{\alpha}(x,y)}{\partial y_{2}}ds(y)\\
    &&=\frac{i}{2\Lambda}\int_{\Gamma_{-h}}(\sum_{n\in\mathbb{Z}}\frac{1}{\beta_{n}}e^{i\alpha_{n}(x_{1}-y_{1})+i\beta_{n}(h+x_{2})}\cdot\sum_{m\in\mathbb{Z}}-i\beta_{m}u_{m,\alpha}^{s+}e^{i\alpha_{m}y_{1}-i\beta_{m}h}\\
    &&-\sum_{n\in\mathbb{Z}}-ie^{i\alpha_{n}(x_{1}-y_{1})+i\beta_{n}(x_{2}+h)}\cdot\sum_{m\in\mathbb{Z}}u_{m,\alpha}^{s+}e^{i\alpha_{m}y_{1}-i\beta_{m}h})ds(y)\\
    &&=0.
\end{eqnarray*}
Thus, $I_{1}=0$, for $I_{2}$, 
\begin{eqnarray*}
    I_{2}=-\int_{\partial B_{\rho}(x)}G^{qp}_{\alpha}(x,y)\frac{\partial u^{s}_{\alpha}(y)}{\partial\nu(y)}-u^{s}_{\alpha}(y)\frac{\partial G^{qp}_{\alpha}(x,y)}{\partial\nu(y)}ds(y).
\end{eqnarray*}
Recalling \eqref{qpgreen}, the singular term of $G_{\alpha}^{qp}(x,y), y\in\partial B_{\rho}(x)$ as $\rho\rightarrow 0$ is $G(x,y)=\frac{i}{4}H_{0}^{1}(k|x-y|)$. Thus, 
\begin{eqnarray*}
    I_{2}=-\int_{\partial B_{\rho}(x)}G(x,y)\frac{\partial u^{s}_{\alpha}(y)}{\partial\nu(y)}-u^{s}_{\alpha}(y)\frac{\partial G(x,y)}{\partial\nu(y)}ds(y).
\end{eqnarray*}
By the singularity of the fundamental solution to free-space Green's function, letting $\rho\rightarrow 0$, we have $I_{2}=-u_{\alpha}^{s}(x)$
\end{proof}
\end{theorem}
Further, we introduce the following function spaces.
\begin{eqnarray*}
    H^{1}_{\alpha,qp}(\Omega)=\lbrace u\in H^{1}(\Omega)|u(x_{1}+\Lambda,x_{2})e^{-i\Lambda\alpha}=u(x_{1},x_{2})\rbrace,
\end{eqnarray*}
with induced norm from $H^{1}(\Omega)$,and
\begin{eqnarray*}
    L^{2}_{\alpha,qp}(\Omega)=\lbrace u\in L^{2}(\Omega)|u(x_{1}+\Lambda,x_{2})e^{-i\Lambda\alpha}=u(x_{1},x_{2})\rbrace,
\end{eqnarray*} with induced norm from $L^{2}(\Omega)$
Now we further include the well-posedness of a $\alpha$ quasi-periodic scattering solution to the Helmholtz equation, whose proof is similar to that of \cite{Yang2}.
\begin{theorem}\label{thm:2.5}
    Assume that $\gamma(x)-1\in L^{\infty}(\Omega)$ is periodic in $x_{1}$ and $f(x)\in L^{2}_{\alpha,qp}(\Omega)$ is $\alpha$-quasi-periodic in $x_{1}$, with period $\Lambda$. Both of them are supported in $D$, for $x\in\Omega$. $w^{s}_{\alpha}$ is the $\alpha$-quasi-periodic scattering solution to the the Helmholtz equation 
\begin{eqnarray*}
    \Delta w_{\alpha}^{s}(x)+k^{2}\gamma(x)w_{\alpha}^{s}(x)=f(x)
\end{eqnarray*}
with Rayleigh expansion condition, namely, 
\begin{eqnarray}
w_{\alpha}^{s}=\left\{\begin{aligned}
\sum_{m\in\mathbb{Z}}w_{m,\alpha}^{s+}e^{i\alpha_{m}x_{1}+i\beta_{m}x_{2}},x_2\in \Omega_{H}^{+},\\
\sum_{m\in\mathbb{Z}}w_{m,\alpha}^{s-}e^{i\alpha_{m}x_{1}-i\beta_{m}x_{2}},x_2\in \Omega_{H}^{-},
\end{aligned}\right.\label{wexpan1}
\end{eqnarray}
Then we have, for some constant $C$ that is dependent on $k,|D|$,
\begin{eqnarray*}
    ||w_{\alpha}^{s}||_{H^{1}_{\alpha,qp}(\Omega)}\leq C||f||_{L^{2}_{p}(D)}
\end{eqnarray*}
\end{theorem}
The final ingredient to our resolution analysis is that the propagating part of the wave carries the major contribution of the cross-correlation.

\begin{theorem}\label{thm:2.6}
Given a compactly supported region $D\subset\Omega$. For any $w^{s}_{\alpha}\in H^{1}_{\alpha,qp}(\Omega\setminus D)$ satisfies the Helmholtz equation 
\begin{eqnarray*}
    \Delta w_{\alpha}^{s}(x)+k^{2}w_{\alpha}^{s}(x)=0,
\end{eqnarray*}
in $\Omega\setminus D$, and the Rayleigh expansion condition \eqref{wexpan1}, taking the clock-wise direction, we have
then we have
\begin{eqnarray}
-Im(\int_{\partial D}\overline{\frac{\partial w^{s}_{\alpha}(y)}{\partial v(y)}}w^{s}_{\alpha}(y)ds(y))=\Lambda\sum_{n\in B_{\alpha}}\beta_{n}(|w_{n,\alpha}^{s+}|^2+|w_{n,\alpha}^{s-}|^2)\label{RLRes1}
\end{eqnarray}
\end{theorem}
\begin{proof}
Take $\Gamma_{h}$ and $\Gamma_{-h}$ that are the line sections above and below $H$, $-H$ in the Rayleigh expansion condition.

We take a rectangular region $\Omega_{h}$ containing $D$, with $\Gamma_{h},\Gamma_{-h}$ the upper and lower horizontal part, while the left and right boundaries being located on $\partial \Omega_{L}$ and $\partial\Omega_{R}$. Using partial integration, and that $\Delta w_{\alpha}^{s}+k^{2}w_{\alpha}^{s}=0$ on $\Omega_{h}\setminus D$, we have
\begin{eqnarray}
\int_{\partial D}\overline{\frac{\partial w_{\alpha}^{s}(y)}{\partial v(y)}}w_{\alpha}^{s}(y)ds(y)=\int_{\Omega_{h}\setminus D}|\nabla w_{\alpha}^{s}|^{2}-k^{2}|w_{\alpha}^{s}|^{2}dy+\int_{\partial \Omega_{h}}\overline{\frac{\partial w_{\alpha}^{s}(y)}{\partial v(y)}}w_{\alpha}^{s}(y)ds(y).
\end{eqnarray}
Namely,
\begin{eqnarray*}
Im\int_{\partial D}\overline{\frac{\partial w_{\alpha}^{s}(y)}{\partial v(y)}}w_{\alpha}^{s}(y)ds(y)=Im(\int_{\partial \Omega_{h}}\overline{\frac{\partial w_{\alpha}^{s}(y)}{\partial v(y)}}w_{\alpha}^{s}(y)ds(y)).
\end{eqnarray*} 
For the integral on the right hand side, since $w_{\alpha}^{s}(y)$ is $\alpha$-quasi-periodic, the integrand inside the integral is periodic in $y_{1}$ direction, thus the left hand part and the right hand part of the section vanishes due to quasi-periodicity. Then
\begin{eqnarray*}
Im\int_{\partial D}\overline{\frac{\partial w_{\alpha}^{s}(y)}{\partial v(y)}}w_{\alpha}^{s}(y)ds(y)=Im(\int_{ \Gamma_{h}\cup \Gamma_{-h}}\overline{\frac{\partial w_{\alpha}^{s}(y)}{\partial v(y)}}w_{\alpha}^{s}(y)ds(y)).
\end{eqnarray*} 
For the integral on $\Gamma_{h}$, since $\nu(y)=(0,1)$using the Rayleigh expansion, we obtain that

\begin{eqnarray*}
-Im(\int_{\Gamma_{h}}\overline{\frac{\partial w^{s}_{\alpha}(y)}{\partial v(y)}}w^{s}_{\alpha}(y)ds(y))=-Im(\int_{\Gamma_{h}}\overline{\frac{\partial w^{s}_{\alpha}(y)}{\partial y_{2}}}w^{s}_{\alpha}(y)ds(y))\\
=\Lambda Im(\sum_{n\in\mathbb{Z}}\beta_{n}|w_{n,\alpha}^{s+}|^2)=\Lambda\sum_{n\in B_{\alpha}}\beta_{n}(|w_{n,\alpha}^{s+}|^2)
\end{eqnarray*}
Where the last equality is due to the fact that $\beta_{n}$ are real if $n\in B_{\alpha}$. Similarly, noticing the change of direction in $\nu(y)=(0,-1)$, we have
\begin{eqnarray*}
-Im(\int_{\Gamma_{-h}}\overline{\frac{\partial w^{s}_{\alpha}(y)}{\partial v(y)}}w^{s}_{\alpha}(y)ds(y))=Im(\int_{-\frac{\Lambda}{2}}^{\frac{\Lambda}{2}}\overline{\frac{\partial w^{s}_{\alpha}(y_{1},-h)}{\partial y_{2}}}w^{s}_{\alpha}(y)ds(y))=\Lambda\sum_{n\in B_{\alpha}}\beta_{n}(|w_{n,\alpha}^{s-}|^2),
\end{eqnarray*}
This completes the proof. 
\end{proof}
\section{The RTM method}\label{sect3}
We are now ready to propose the following RTM imaging functionals
\begin{eqnarray}
\mathcal{I}_U(z)=Im\sum_{n\in B_{\alpha}}\int_{\Gamma_{h}}\frac{i}{\beta_{n}}u^{inc}_{\alpha_{n}}(z)\frac{\partial G^{qp}_{\alpha}(x_{r},z)}{\partial x_{2}}\overline{u^{s}_{\alpha_{n}}(x_{r})}ds(x_{r}),\label{RTMfunc}
\end{eqnarray}
which is named the upper RTM functional, and
\begin{eqnarray}
\mathcal{I}_L(z)=-Im\sum_{n\in B_{\alpha}}\int_{\Gamma_{-h}}\frac{i}{\beta_{n}}u^{inc}_{\alpha_{n}}(z)\frac{\partial G^{qp}_{\alpha}(x_{r},z)}{\partial x_{2}}\overline{u^{s}_{\alpha_{n}}(x_{r})}ds(x_{r}),\label{RTMfuncL}
\end{eqnarray}
which is named the lower RTM functional. We remark that, due to the $\alpha$ quasi-periodicity of $u^{inc}_{\alpha_{n}}$ and the $-\alpha$ quasi-periodicity of $G_{\alpha}^{qp}$, both $\mathcal{I}_{U}$ and $\mathcal{I}_{L}$ are naturally periodic in $z_{1}$ direction. Here,
we take $\mathcal{I}_{L}(z)$ for instance, explain the functional as a two-step algorithm:
\begin{Algorithm}
Given the data $u^{s}_{\alpha_{n}}(x_{r})$, which is the measurement of scattered field on $\Gamma_{h}$, at points $x_{r}\in\Gamma_{-h}=\{(x_{1},-h)|x_{1}\in(-\frac{\Lambda}{2},\frac{\Lambda}{2}), h>0\}, r=1,\cdots,N_{r}$ for all $\alpha_{n}\in B_{\alpha}$,
\begin{enumerate}
\item Back propagation: For all $r=1,\cdots,N_{r}$, $\alpha_{n},n\in B_{\alpha}$, compute 
\begin{eqnarray*}
v_{\alpha_{n}}(z)=\frac{|\Gamma_{r}|}{N_{r}}\sum_{r=1}^{N_{r}}\frac{\partial G^{qp}_{\alpha}(x_{r},z)}{\partial x_{2}}\overline{u^{s}_{\alpha_{n}}(x_{r})}.
\end{eqnarray*}
\item Cross-correlation: For all $z\in\Omega$, calculate:
\begin{eqnarray*}
\mathcal{\hat{I}}_{L}(z)=-Im\sum_{n\in B_{\alpha}}\frac{i}{\beta_{n}}u^{inc}_{\alpha_{n}}(z)v_{\alpha_{n}}(z).
\end{eqnarray*}
\end{enumerate}
\end{Algorithm}
If we use the measurement data on $\Gamma_h$, we will have the upper RTM algorithm similarly. It is seen that $\hat{\mathcal{I}}_{L}(z)$ is an approximation of the continuous integral \eqref{RTMfunc}. 
\section{Resolution analysis}\label{sect4}
In this section we analyze the resolution of the proposed RTM methods. Firstly, we consider the resolution for lower RTM algorithm.
\subsection{The resolution for lower RTM}
\begin{theorem}\label{thm:4.1}
The lower RTM functional has the following resolution analysis
\begin{eqnarray}
\mathcal{I}_{L}(z)=\Lambda^2 Im\int_{D}k^{2}(1-\gamma(y))(\overline{F^{L}_{\alpha}(y,z)}+\overline{v^{s}_{\alpha}(y,z)})F^{L}_{\alpha}(y,z)dy +O(h^{-1}).\label{Il}
\end{eqnarray}
Where $v_{n,\alpha}^{s}$, are the Rayleigh coefficients to the  scattering solution $v^{s}_{\alpha}$ to
\begin{equation}
\Delta_{y} v_{\alpha}^{s}(y,z)+k^{2}\gamma(y)v_{\alpha}^{s}(y,z)=k^{2}(1-\gamma(y))F^{L}_{\alpha}(y,z)\label{AdjEq1}
\end{equation}
with the Rayleigh scattering condition, for $n\in B_{\alpha}$. Further,in terms of the Rayleigh-coefficients, we have
\begin{equation}
\mathcal{I}_{L}(z)=\Lambda^{2}\sum_{n\in B_{\alpha}}\beta_{n}(|v^{s+}_{n,\alpha}(z)|^{2}+|v^{s-}_{n,\alpha}(z)|^{2})+O(h^{-1}).\label{RTMRes1}
\end{equation}
\end{theorem}
\begin{proof}
Recalling the lower RTM functional for all $z\in\Omega^+_{-h}$
\begin{equation}
\mathcal{I}_{L}(z)=-Im\sum_{n\in B_{\alpha}}\int_{\Gamma_{-h}}\frac{i}{\beta_{n}}u_{\alpha_{n}}^{inc}(z)\frac{\partial G^{qp}_{\alpha}(x_{r},z)}{\partial x_{2}}\overline{u_{\alpha_{n}}^{s}(x_{r})}ds(x_{r}).
\end{equation} 
Using Theorem \ref{thm:2.3}, and Corollary \ref{cor:2.1}, we have
\begin{eqnarray*}
&\int_{\Gamma_{-h}}&\frac{\partial G^{qp}_{\alpha}(x_{r},z)}{\partial x_{2}}\overline{u_{\alpha_{n}}^{s}(x_{r})}ds(x_{r})\\
=&\int_{\Gamma_{-h}}&\frac{\partial G^{qp}_{\alpha}(x_{r},z)}{\partial x_{2}}(\int_{D}k^{2}(\gamma(y)-1)\overline{u_{\alpha_{n}}(y)}\overline{G^{qp}_{\alpha}(x_{r},y)}dy)ds(x_{r})\\
=&\int_{D}&k^{2}(\gamma(y)-1)\overline{u_{\alpha_{n}}(y)}\int_{\Gamma_{-h}}\frac{\partial G^{qp}_{\alpha}(x_{r},z)}{\partial x_{2}}\overline{G^{qp}_{\alpha}(x_{r},y)}ds(x_{r})dy\\
=&\int_{D}&k^{2}(\gamma(y)-1)\overline{u_{\alpha_{n}}(y)}(\frac{1}{2}\overline{F^{L}_{\alpha}(z,y)}+\overline{R^{L}_{\alpha}(z,y;h)})dy.
\end{eqnarray*}
Here $|R^{L}_{\alpha}(z,y;h)|\leq C(h^{-1})$.
we obtain that
\begin{eqnarray*}
\mathcal{I}_{L}(z)=\Lambda Im\int_{D}k^{2}(1-\gamma(y))\overline{F_{\alpha}^{L}(z,y)}\sum_{n\in B_{\alpha}}(\frac{i}{2\Lambda\beta_{n}}u^{inc}_{\alpha_{n}}(z)\overline{u_{\alpha_{n}}(y)})dy  + O(h^{-1}).
\end{eqnarray*}
Now we introduce
\begin{eqnarray*}
v_{\alpha}(y,z) &=& \sum_{n\in B_{\alpha}}\frac{i}{2\Lambda\beta_{n}}u_{\alpha_{n}}(y) \overline{u^{inc}_{\alpha_{n}}(z)},
\end{eqnarray*}
it follows that
\begin{eqnarray*}
    \sum_{n\in B_{\alpha}}\frac{i}{2\Lambda\beta_{n}}u^{inc}_{\alpha_{n}}(z)\overline{u_{\alpha_n}(y)}=-\overline{ \sum_{n\in B_{\alpha}}\frac{i}{2\Lambda\beta_{n}}u_{\alpha_{n}}(y)\overline{u^{inc}_{\alpha_{n}}(z)}}=-\overline{v_{\alpha}(y,z)}.
\end{eqnarray*}
Further, since
\begin{eqnarray*}
\overline{F^{L}_{\alpha}(z,y)}&=&\overline{\sum_{n\in B_{\alpha}}\frac{i}{2\Lambda\beta_{n}}e^{i\alpha_{n}(z_{1}-y_{1})-i\beta_{n}(z_{2}-y_{2})}}\\
&=&-\sum_{n\in B_{\alpha}}\frac{i}{2\Lambda\beta_{n}}e^{i\alpha_{n}(y_{1}-z_{1})-i\beta_{n}(y_{2}-z_{2})}\\
&=&-F^{L}_{\alpha}(y,z).
\end{eqnarray*} 
It follows that
\begin{eqnarray*}
\mathcal{I}_{L}(z)=\Lambda Im\int_{D}k^{2}(1-\gamma(y))F_{\alpha}^{L}(y,z)\overline{v_{\alpha}(y,z)}dy  + O(h^{-1}).
\end{eqnarray*}
Recalling that 
\begin{eqnarray*}
u_{\alpha_{n}}(y)=u^{inc}_{\alpha_{n}}(y)+u^{s}_{\alpha_{n}}(y).
\end{eqnarray*}
If we further introduce
\begin{eqnarray*}
v^{s}_{\alpha}(y,z)&=&v_{\alpha}(y,z)-F^{L}_{\alpha}(y,z)\\
&=&\sum_{n\in B_{\alpha}}\frac{i}{2\Lambda\beta_{n}}u_{\alpha_{n}}(y)\overline{u^{inc}_{\alpha_{n}}(z)}-\sum_{n\in B_{\alpha}}\frac{i}{2\Lambda\beta_{n}}u_{\alpha_{n}}^{inc}(y)\overline{u^{inc}_{\alpha_{n}}(z)}\\
&=&\sum_{n\in B_{\alpha}}\frac{i}{2\Lambda\beta_{n}}u_{\alpha_{n}}^{s}(y)\overline{u^{inc}_{\alpha_{n}}(z)}.
\end{eqnarray*}
We observe that it is the solution to 
\begin{eqnarray*}
\Delta_{y} v_{\alpha}^{s}(y,z)+k^{2}v_{\alpha}^{s}(y,z)=k^{2}(1-\gamma(y))v_{\alpha}(y,z).
\end{eqnarray*}
Thus it is the $\alpha$-quasi-periodic solution to
\begin{eqnarray*}
\Delta v_{\alpha}^{s}(y,z)+k^{2}\gamma(y)v_{\alpha}^{s}(y,z)=k^{2}(1-\gamma(y))F^{L}_{\alpha}(y,z) 
\end{eqnarray*}
and satisfies the Rayleigh expansion condition,
\begin{eqnarray*}
v_{\alpha}^{s}(y,z)=\left\{\begin{aligned}
\sum_{m\in\mathbb{Z}}v_{m,\alpha}^{s+}(z)e^{i\alpha_{m}y_{1}+i\beta_{m}y_{2}},y_2\in \Omega_{H}^{+},\\
\sum_{m\in\mathbb{Z}}v_{m,\alpha}^{s-}(z)e^{i\alpha_{m}y_{1}-i\beta_{m}y_{2}},y_2\in \Omega_{H}^{-}.
\end{aligned}
\right.
\end{eqnarray*}
Eventually, we have
\begin{eqnarray}
\mathcal{I}_{L}(z)=\Lambda Im\int_{D}k^{2}(1-\gamma(y))(\overline{F^{L}_{\alpha}(y,z)}+\overline{v^{s}_{\alpha}(y,z)})F^{L}_{\alpha}(y,z)dy +O(h^{-1}),\label{res1}
\end{eqnarray}
which is
\begin{eqnarray*}
\mathcal{I}_{L}(z)=\Lambda Im\int_{D}k^{2}(1-\gamma(y))\overline{v_{\alpha}^{s}(y,z)}F^{L}_{\alpha}(y,z)dy +O(h^{-1}).
\end{eqnarray*}
Since that 
\begin{eqnarray}
&&Im\int_{D} k^{2}(1-\gamma(y))F^{L}_{\alpha}(y,z)\overline{v^{s}_{\alpha}(y,z)}dy
=Im\int_{D}(\Delta v_{\alpha}^{s}(y,z)+k^{2}\gamma(y)v_{\alpha}^{s}(y,z))\overline{v_{\alpha}^{s}(y,z)}dy\nonumber\\
&&\qquad\qquad=Im\int_{D} \Delta v_{\alpha}^{s}(y,z)\overline{v_{\alpha}^{s}(y,z)}dy\nonumber\\
&&\qquad\qquad=-Im\int_{\partial D}\frac{\partial \overline {v_{\alpha}^{s}(y,z)}}{\partial \nu (y)}v_{\alpha}^{s}(y,z)ds.\label{bintegral}
\end{eqnarray}
By Theorem \ref{thm:2.6}, we obtain that 
\begin{eqnarray*}
\mathcal{I}_{L}(z)&=&-\Lambda Im\int_{\partial D}\frac{\partial \overline {v_{\alpha}^{s}(y,z)}}{\partial \nu (y)}v_{\alpha}^{s}(y,z)ds(y)+O(h^{-1})\\
&=&\Lambda^2\sum_{n\in B_{\alpha}}\beta_{n}(|v^{s+}_{n,\alpha}(z)|^{2}+|v^{s-}_{n,\alpha}(z)|^{2})+O(h^{-1}).
\end{eqnarray*}
\end{proof}
\subsection{The resolution for upper RTM}
For $\mathcal{I}_{U}(z)$, following similar steps as the proof of Theorem \ref{thm:4.1} until \eqref{res1}, we arrive at the following result.
\begin{theorem}\label{thm:4.2}
The upper RTM functional has the following representation
\begin{eqnarray}
\mathcal{I}_{U}(z)=\Lambda Im\int_{D}k^{2}(1-\gamma(y))(\overline{F^{L}_{\alpha}(y,z)}+\overline{v^{s}_{\alpha}(y,z)})F^{U}_{\alpha}(y,z)dy +O(h^{-1}).\label{Iu}
\end{eqnarray}
Where $v_{\alpha}^{s}$, is the $\alpha$-quasi-periodic scattering solution 
\begin{eqnarray}
\Delta_{y} v_{\alpha}^{s}(y,z)+k^{2}\gamma(y)v_{\alpha}^{s}(y,z)=k^{2}(1-\gamma (y))F^{L}_{\alpha}(y,z)
\end{eqnarray}
with the Rayleigh expansion condition.
\end{theorem}
From \eqref{Iu}, we see that the decaying property of $Im(F^{L}_{\alpha})$ and $Im(F^{U}_{\alpha})$ as $z$  leaves $\partial D$, gives that $\mathcal{I}_{U}(z)$ has the decaying property as $z$ leaves $\partial D$, in the probing area $\Omega_{0}$.
On the other hand, as $z\rightarrow\partial D$, we denote the main part of $\mathcal{I}_U(z)$ by
\begin{eqnarray*}
    \tilde{\mathcal{I}}_{U}(z)=\Lambda Im\int_{D}k^{2}(1-\gamma(y))(\overline{F^{L}_{\alpha}(y,z)}+\overline{v^{s}_{\alpha}(y,z)})F^{U}_{\alpha}(y,z)dy. \label{Imajor}
\end{eqnarray*}
Recalling
\begin{eqnarray*}
F^{U}_{\alpha}(y,z)=F^{1}_{\alpha}(y,z)+iF^{2}_{\alpha}(y,z),\quad F^{L}_{\alpha}(y,z)=F^{1}_{\alpha}(y,z)-iF^{2}_{\alpha}(y,z).\label{expan}
\end{eqnarray*}
We let 
$v^{i}_{\alpha}(y,z),i=1,2, n\in B_{\alpha}$, be the $\alpha$-quasi-periodic scattering solution that corresponds to the equations
\begin{eqnarray*}
\Delta v^{i}_{\alpha}(y,z) + k^2 \gamma(y)v^{i}_{\alpha}(y,z)=k^{2}(1-\gamma(y))F^{i}_{\alpha}(y,z),
\end{eqnarray*} 
and satisfy the Rayleigh expansion condition
\begin{eqnarray*}
v_{\alpha}^{i}(y,z)=\left\{\begin{aligned}
\sum_{m\in\mathbb{Z}}v_{m,\alpha}^{i+}(z)e^{i\alpha_{m}y_{1}+i\beta_{m}y_{2}},y_2\in \Omega_{H}^{+},\\
\sum_{m\in\mathbb{Z}}v_{m,\alpha}^{i-}(z)e^{-i\alpha_{m}y_{1}+i\beta_{m}y_{2}},y_2\in \Omega_{H}^{-}.
\end{aligned}\right.
\end{eqnarray*}
It follows that
\begin{eqnarray}
\tilde{\mathcal{I}}_{U}(z)&=&\Lambda Im\int_{D}k^{2}(1-\gamma(y))(\overline{F^{1}_{\alpha}}+i\overline{F^{2}_{\alpha}}+\overline{v_{\alpha}^{1}}+i\overline{v_{\alpha}^{2}})(F^{1}_{\alpha}+iF^{2}_{\alpha})dy\nonumber \\
&=&\Lambda Im\int_{D}k^{2}(1-\gamma(y))F_{\alpha}^{1}\overline{v_{\alpha}^{1}}dy+R_{1}(z)+ R_{2}(z).\label{Iurtm}
\end{eqnarray}
Here we have
\begin{eqnarray}
R_{1}(z)=-\Lambda Im\int_{D}(1-\gamma(y))F_{\alpha}^{2}(y,z)\overline{v_{\alpha}^{2}(y,z)}dy,\label{IurtmR1}
\end{eqnarray}
and
\begin{eqnarray}
R_{2}(z)=\Lambda Im\int_{D}ik^{2}(1-\gamma(y))(\overline{v_{\alpha}^{2}}F_{\alpha}^{1}+\overline{v_{\alpha}^{1}}F_{\alpha}^{2}+\overline{F_{\alpha}^1}F_{\alpha}^{2}+\overline{F_{\alpha}^{2}}F_{\alpha}^{1})dy.
\label{IurtmR2}\end{eqnarray}
The definition of $F_{\alpha}^{2}$ \eqref{Falpha2} indicates that $F_{\alpha}^{2}=O(|z_{2}-y_{2}|)$ as $y\rightarrow z$.  Then with the help of Theorem \ref{thm:2.5}, we have $v_{\alpha}^{2}=O(|z_{2}-y_{2}|)$ as $y\rightarrow z$. Since the integral 
can be converted to an integral on $\partial D$ as \eqref{bintegral} in the proof of Theorem 4.1, we know that $R_{1}(z),R_{2}(z)\rightarrow 0$ as $z$ approaches $\partial D$. 

Thus the property of $\mathcal{I}_{U}(z)$ as $z$ approaches the boundary of $D$ is reflected in $v_{\alpha}^{1\pm}(z)$, which peak at the boundary with similar behavior as that of $Im(F^{1}_{\alpha}(z))$. 

\section{Extensions to sound-soft case}\label{sect5}
Our RTM functionals $\mathcal{I}_{L}(z)$, $\mathcal{I}_{U}(z)$ can also be applied to the case of detecting sound soft periodic array. Namely, we are given $u^{inc}_{\alpha_{n}}$ as above, and the scattered field $u^s_{\alpha_{n}}$ is given by
\begin{eqnarray*}
\Delta u^s_{\alpha_{n}} + k^2 u^s_{\alpha_{n}} = 0  \mbox{ in } \mathbb{R}^2\setminus\bar{D},\\
u^s_{\alpha_{n}}=-u^{inc}_{\alpha_n}\textit{ on $\partial D$},
\end{eqnarray*}
with $\alpha_{n}$ quasi-periodicity in $x_{1}$ direction, and satisfies the Rayleigh-expansion condition. 
We can obtain the following resolution results on the two RTM functionals for the case of sound-soft periodic scattering problem. 
\begin{theorem}
    Let the lower RTM functional be given by \eqref{RTMfuncL}, and let $\psi(y,z)$ be the solution to the following problem
    \begin{eqnarray}
        \Delta_y \psi(y,z) + k^2 \psi(y,z) = 0, \textit{ for $ y\in \mathbb{R}^{2}\setminus \bar{D}$ , $\psi(y,z) = -F^{L}_{\alpha}(y,z)$, for $y\in\partial D$},\label{ssoftl}
    \end{eqnarray}
    with the Rayleigh expansion condition
    \begin{eqnarray*}
\psi(y,z)=\left\{\begin{aligned}
\sum_{m\in\mathbb{Z}}\psi^{+}_{n}(z)e^{i\alpha_{m}y_{1}+i\beta_{m}y_{2}},y_2\in \Omega_{H}^{+},\\
\sum_{m\in\mathbb{Z}}\psi^{-}_{n}(z)e^{i\alpha_{m}y_{1}-i\beta_{m}y_{2}},y_2\in \Omega_{H}^{-}.
\end{aligned}
\right.
\end{eqnarray*}
    We have the following result
    \begin{eqnarray}
        \mathcal{I}_{L}(z)=2\Lambda^2\sum_{n\in B_{\alpha}}\beta_{n}(|\psi^{+}_{n}(z)|^{2}+|\psi^{-}_{n}(z)|^{2})+O(h^{-1})\label{SRTML1}.
    \end{eqnarray}
\end{theorem}
\begin{proof}
    We recall the Green's representation formula for sound-soft obstacle scattering, 
    \begin{eqnarray*}
        u^{s}_{\alpha_{n}}(x_{r})=\int_{\partial D}u^{s}_{\alpha_{n}}(y)\frac{\partial G_{\alpha}^{qp}(x_{r},y)}{\partial \nu(y)}-\frac{\partial u^{s}_{\alpha_{n}}(y)}{\partial\nu(y)}G_{\alpha}^{qp}(x_{r},y)ds(y).
    \end{eqnarray*}
    Thus, using Corollary \ref{cor:2.1}, we obtain
    \begin{eqnarray}
        &\int_{\Gamma_{-h}}&\frac{\partial G(x_{r},z)}{\partial x_{2}}\overline{u^{s}_{\alpha_{n}}(x_{r})}ds(x_{1})\nonumber\\
        &=&\int_{\partial D}\overline{u^{s}_{\alpha_{n}}(y)}\frac{\partial(-\frac{1}{2}F^{L}_{\alpha}(y,z)+\overline{R^{L}_{\alpha}(z,y;h)})}{\partial\nu(y)}\nonumber\\
        &-&\overline{\frac{\partial u^{s}_{\alpha_{n}}(y)}{\partial\nu(y)}}(-\frac{1}{2}F^{L}_{\alpha}(y,z)+\overline{R^{L}_{\alpha}(z,y;h)})ds(y)\nonumber\\
        &=&-\frac{1}{2}\int_{\partial D}\overline{u^{s}_{\alpha_{n}}(y)}\frac{\partial F_{\alpha}^{L}(y,z)}{\partial \nu(y)}-\overline{\frac{\partial u^{s}_{\alpha_{n}}(y)}{\partial\nu(y)}}F_{\alpha}^{L}(y,z)ds(y)+R_{I}(z;h).\nonumber
    \end{eqnarray}
    Here
    \begin{eqnarray*}
        R_{I}(z;h)=\int_{\partial D}\overline{u^{s}_{\alpha_{n}}(y)}\frac{\partial \overline{R^{L}_{\alpha}(z,y;h)}}{\partial \nu(y)}-\overline{\frac{\partial u^{s}_{\alpha_{n}}(y)}{\partial\nu(y)}}\overline{R^{L}_{\alpha}(z,y;h)}ds(y).
\end{eqnarray*}Using corollary \ref{cor:2.1} once again, we obtain that $|R_{I}(z;h)|=O(h^{-1}),h\rightarrow\infty$. Eventually, we obtain that,
    \begin{eqnarray*}
        \mathcal{I}_{L}(z)=-\Lambda Im\int_{\partial D}\overline{\psi(y,z)}\frac{\partial F_{\alpha}^{L}(y,z)}{\partial \nu(y)}-\frac{\overline{\partial\psi(y,z)}}{\partial\nu(y)}F_{\alpha}^{L}(y,z)ds(y)+O(h^{-1}),    \end{eqnarray*}
    where $\psi(y,z)$ is given by \eqref{ssoftl}. Using the sound-soft boundary condition, we obtain that
    \begin{eqnarray}
      \mathcal{I}_{L}(z)=-2\Lambda Im\int_{\partial D}\psi(y,z)\overline{\frac{\partial \psi(y,z)}{\partial \nu(y)}}ds(y)+O(h^{-1}).
    \end{eqnarray}
    Now, using theorem \ref{thm:2.6}, we obtain the desired resolution analysis.
\end{proof}
Using similar technique, we may obtain the analysis for $I_{U}(z)$.
\begin{theorem}
    Let the upper RTM functional be given by \eqref{RTMfunc}, and let $\psi(y,z)$ be the solution to the following problem
    \begin{eqnarray}
        \Delta_y \psi(y,z) + k^2 \psi(y,z) = 0, \textit{ for $ y\in \mathbb{R}^{2}\setminus \bar{D}$ , $\psi(y,z) = -F^{L}_{\alpha}(y,z)$, for $y\in\partial D$}.
    \end{eqnarray}
    We have the following result
    \begin{eqnarray}
        \mathcal{I}_{U}(z)=-\Lambda Im\int_{\partial D}\overline{\psi(y,z)}\frac{\partial F^{U}_{\alpha}(y,z)}{\partial v(y)}-F^{U}_{\alpha}(y,z)\frac{\partial\overline{\psi(y,z)} }{\partial v(y)}dy+O(h^{-1}).\label{SRTMU1}
    \end{eqnarray}
\end{theorem}

\section{Numerical Result}\label{sect6}
In this section, we test several cases of the periodic scattering objects to demonstrate the imaging ability of our imaging functionals $\mathcal{I}_L(z),\mathcal{I}_U(z)$.

The probing area of our numerical experiment is \begin{eqnarray*}
    \Omega_{0}=\{(z_{1},z_{2})||z_{1}|\leq -\frac{\Lambda}{2},|z_{2}|\leq \frac{\Lambda}{2}\}.
\end{eqnarray*} 
In the following experiments, we choose $\Lambda=2\pi$. The probing area is discretized by $101\times 101$ equally distributed points, and the number of receiver on $\Gamma_{\pm h}$, where $h=7$, 
 is $N_{r}=101$. Since the structure of our RTM functionals have periodicity in the $z_{1}$ direction  with period $\Lambda$, the reconstruction in this single period reflects the reconstruction for the periodic array. 
To get the synthetic data of the quasi-periodic scattered wave, we use the MPSPACK 
 based on a modified Nystr\"om method  proposed by \cite{Barnett}. For the calculation of quasi-periodic Green's function explicitly used in the indicator function, we follow the Ewald's method and the procedure introduced in \cite{Ammari} to obtain a fast simulation.

 The refractive index of our numerical experiment for penetrable obstacle is $\gamma(x)=1.5$. The boundaries of the obstacles that are used in our numerical experiments are listed below, where $t\in[0,2\pi)$, $\rho>0$.
 \begin{itemize}
     \item Circle: $z_{1} = \rho \cos(t)$, $z_{2}=\rho \sin(t)$.
     \item Kite: $z_{1} =\rho(1.1\cos(t)+0.625\cos(2t)-0.625)$, $z_{2}=\rho(1.5\sin(t))$.
     \item Peanut: $z_{1} = \cos(t)+\rho \cos(3t)$, $z_{2} = \sin(t)+\rho\sin(3t)$.
 \end{itemize}
We remark that $\alpha = k\cos\theta$ and $\theta$ is the incident angle. 
In our numerical examples, the incoming angles are chosen as $\theta = \frac{\pi}{2}+m\frac{\pi}{16}, m\in \mathbb{Z}$. 
In Example 1 and 2, for Figure \ref{fig3} to Figure \ref{fig6}, $(a)$ corresponds to $\theta=\frac{\pi}{2}$ and $m=0$, that is $\alpha = 0$, which is the vertical incident direction.  $(b)$ corresponds to the average of sum for imaging functionals of 5 different incident angles with $m=0, \pm 1, \pm 2$, while $(c)$ corresponds to the average of sum for imaging functionals of 9 different incident angles with $m=0, \pm 1, \pm 2, \pm 3, \pm 4$.
 
\paragraph{Example 1}\label{ex1}
 In this example, we consider the imaging of penetrable periodic circles with radius $\rho=0.8$ at $k = 5.2\pi$ by our RTM functionals. Figure \ref{fig3} shows the imaging quality of $\mathcal{I}_{L}$, which demonstrates that the imaging functional has positive values and peaks at the boundary of the scatterer. Figure \ref{fig4} shows the imaging results of $\mathcal{I}_{U}$. It is clear from the pictures that one can get better imaging result as the number of $\alpha$ increases. We remark that the imaging result of $\mathcal{I}_{L}$ is sharp especially for the vertical part of the boundary, while for $\mathcal{I}_{U}$, the horizontal part of the circle is imaged clearly.

\begin{figure}[ht]
\subfloat[\label{fig3:a}one $\alpha$]{\includegraphics[width=5cm]{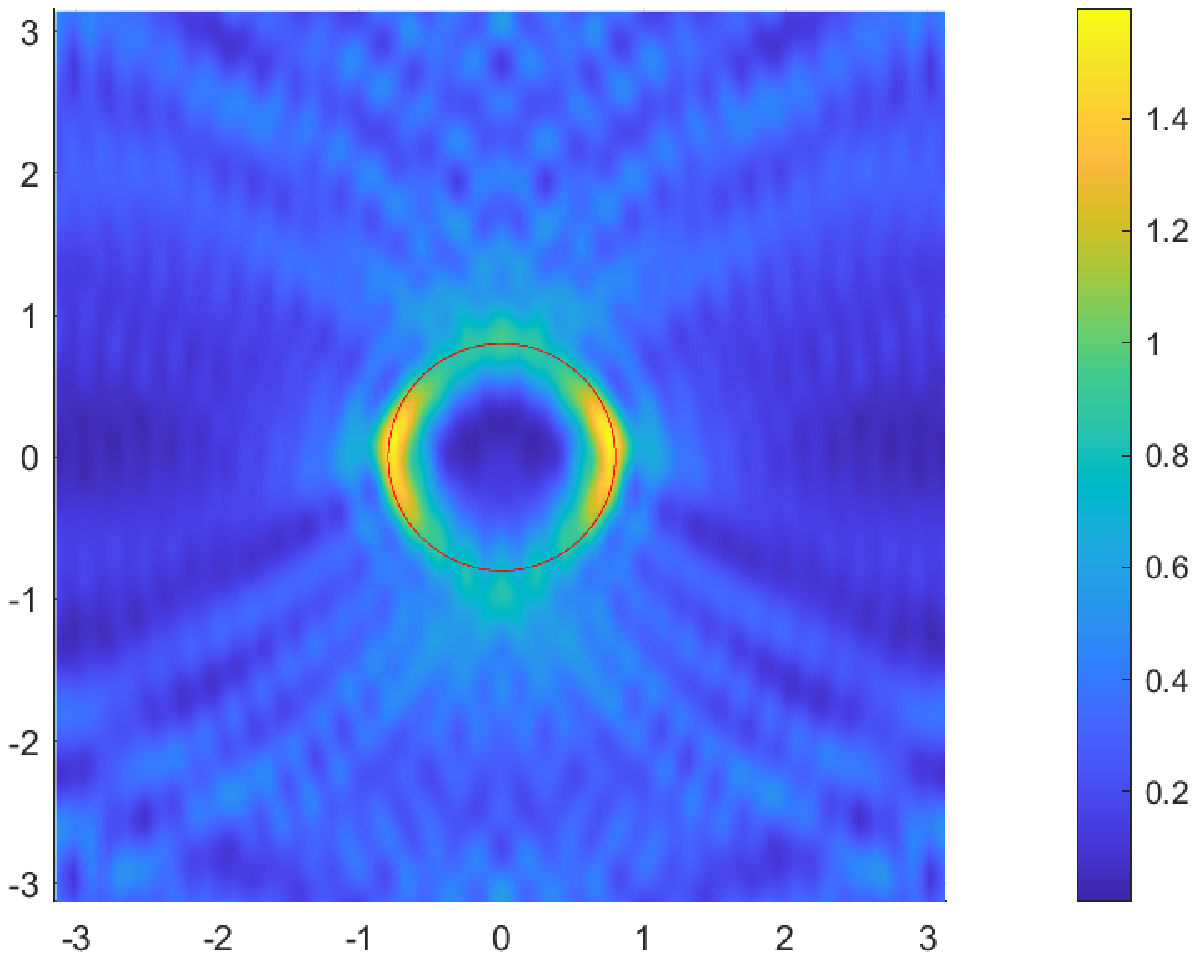}}
\subfloat[\label{fig3:b}five $\alpha$s]{\includegraphics[width=5cm]{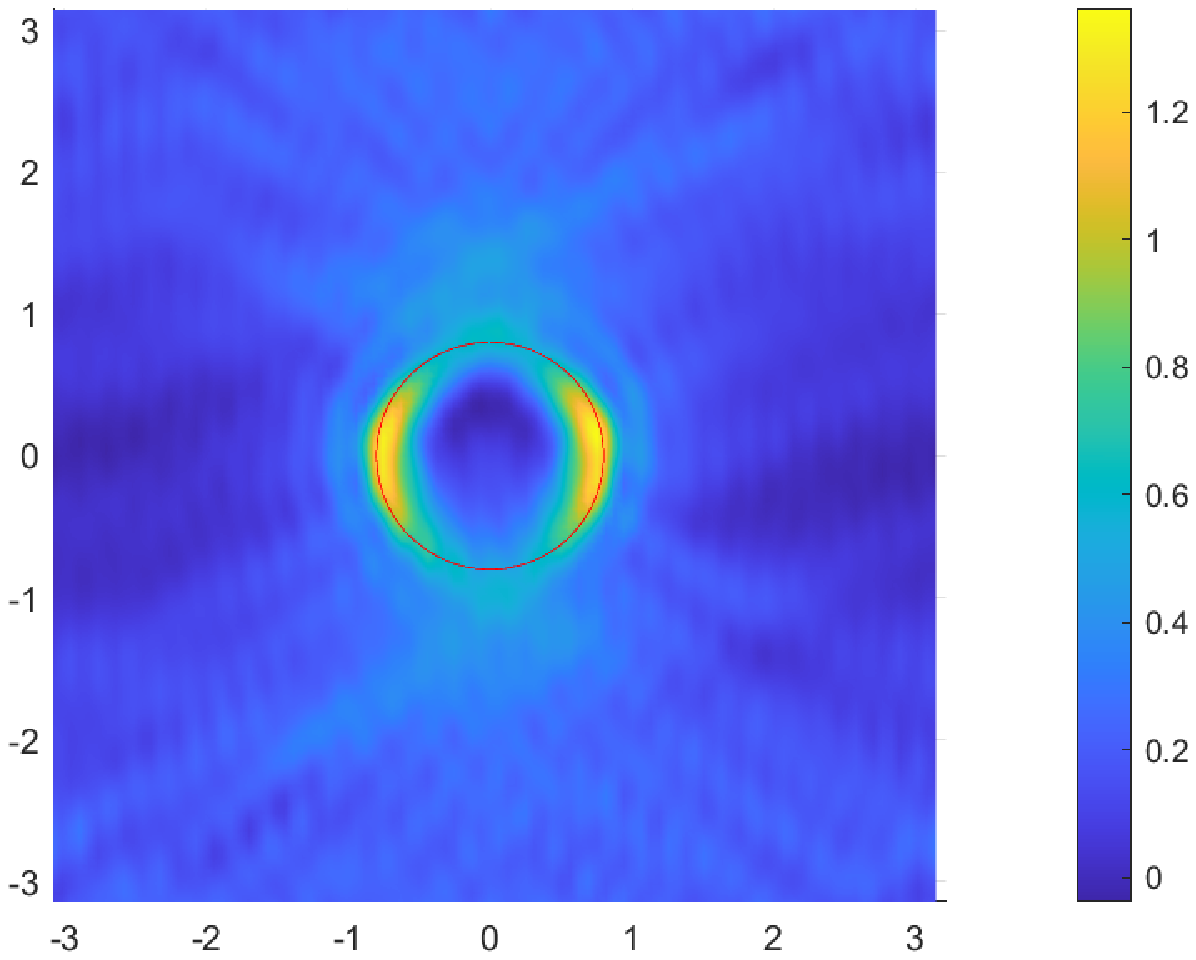}}
\subfloat[\label{fig3:c}nine $\alpha$s]{\includegraphics[width=5cm]{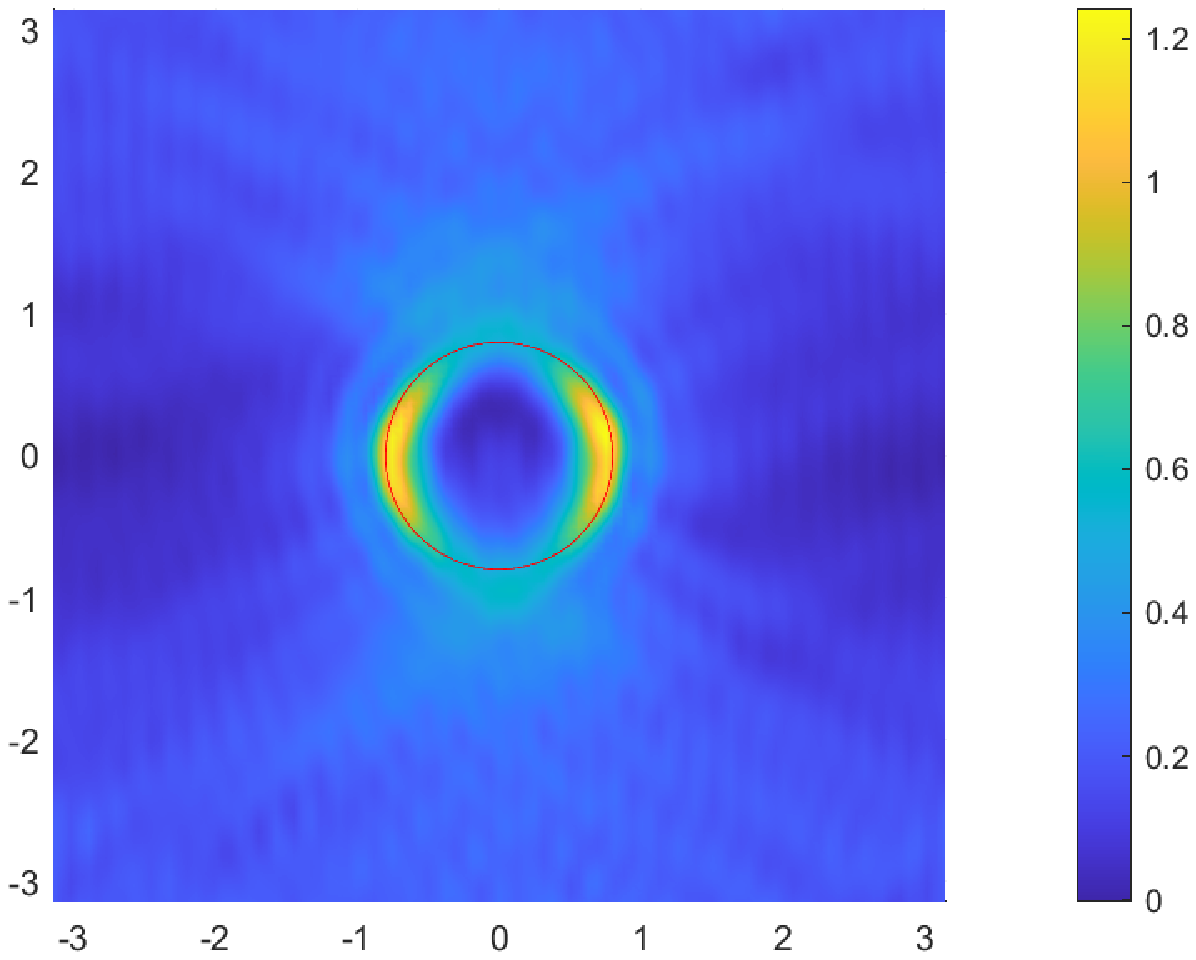}}
\caption{Reconstruction by lower RTM for penetrable circle.}\label{fig3}
\end{figure}
\begin{figure}
\subfloat[\label{fig4:a}one $\alpha$]{\includegraphics[width=5cm]{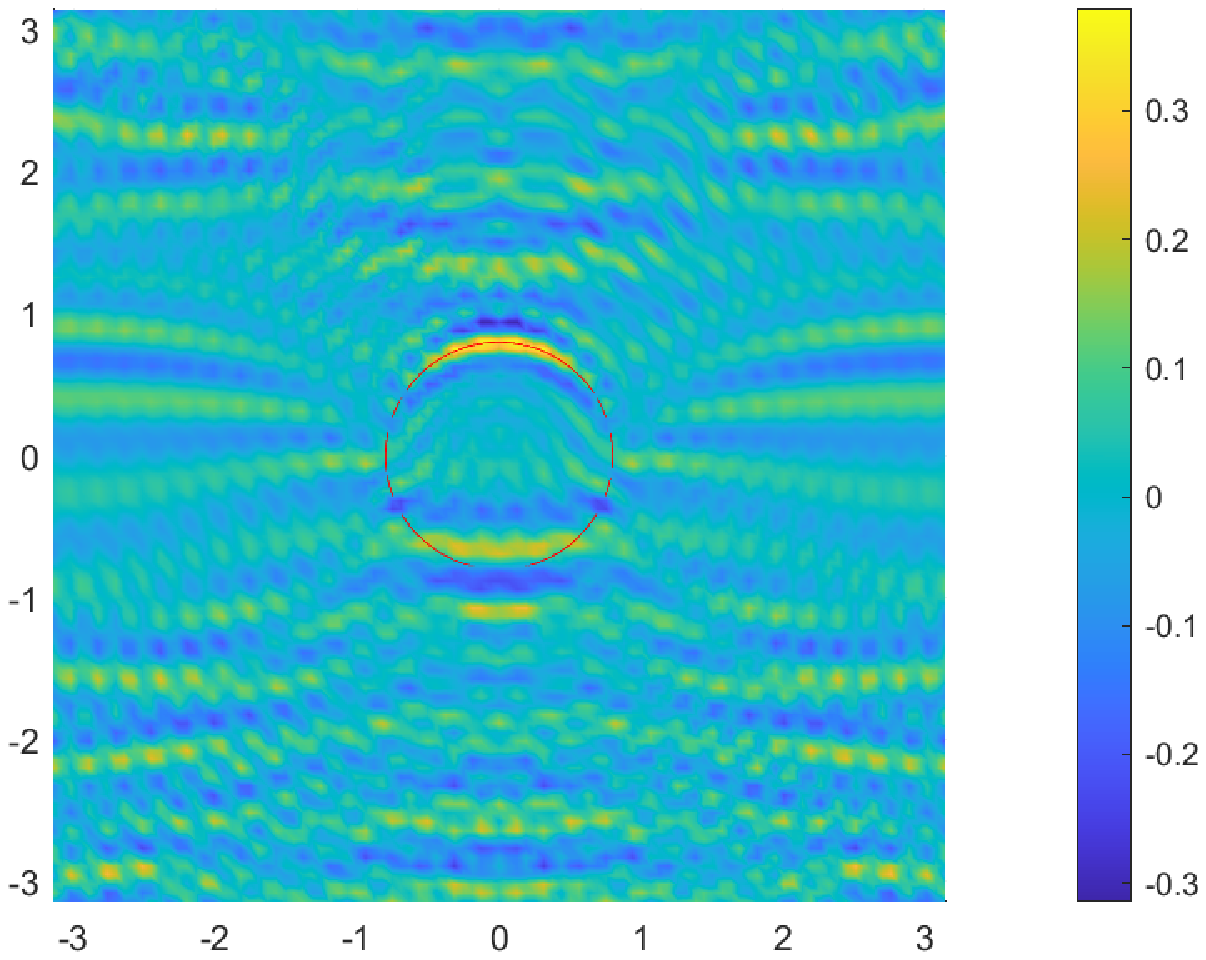}}
\subfloat[\label{fig4:b}five $\alpha$s]{\includegraphics[width=5cm]{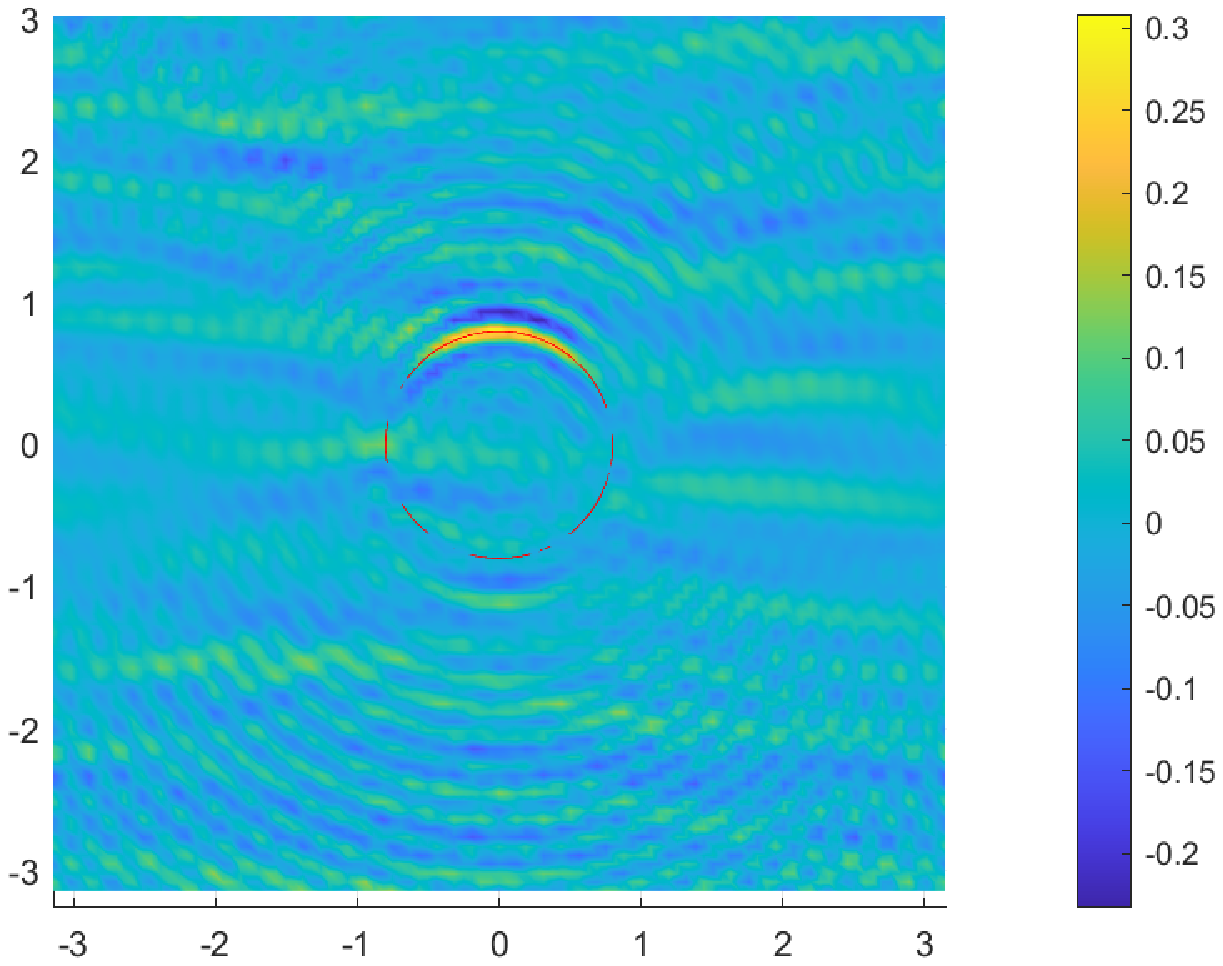}}
\subfloat[\label{fig4:c}nine $\alpha$s]{\includegraphics[width=5cm]{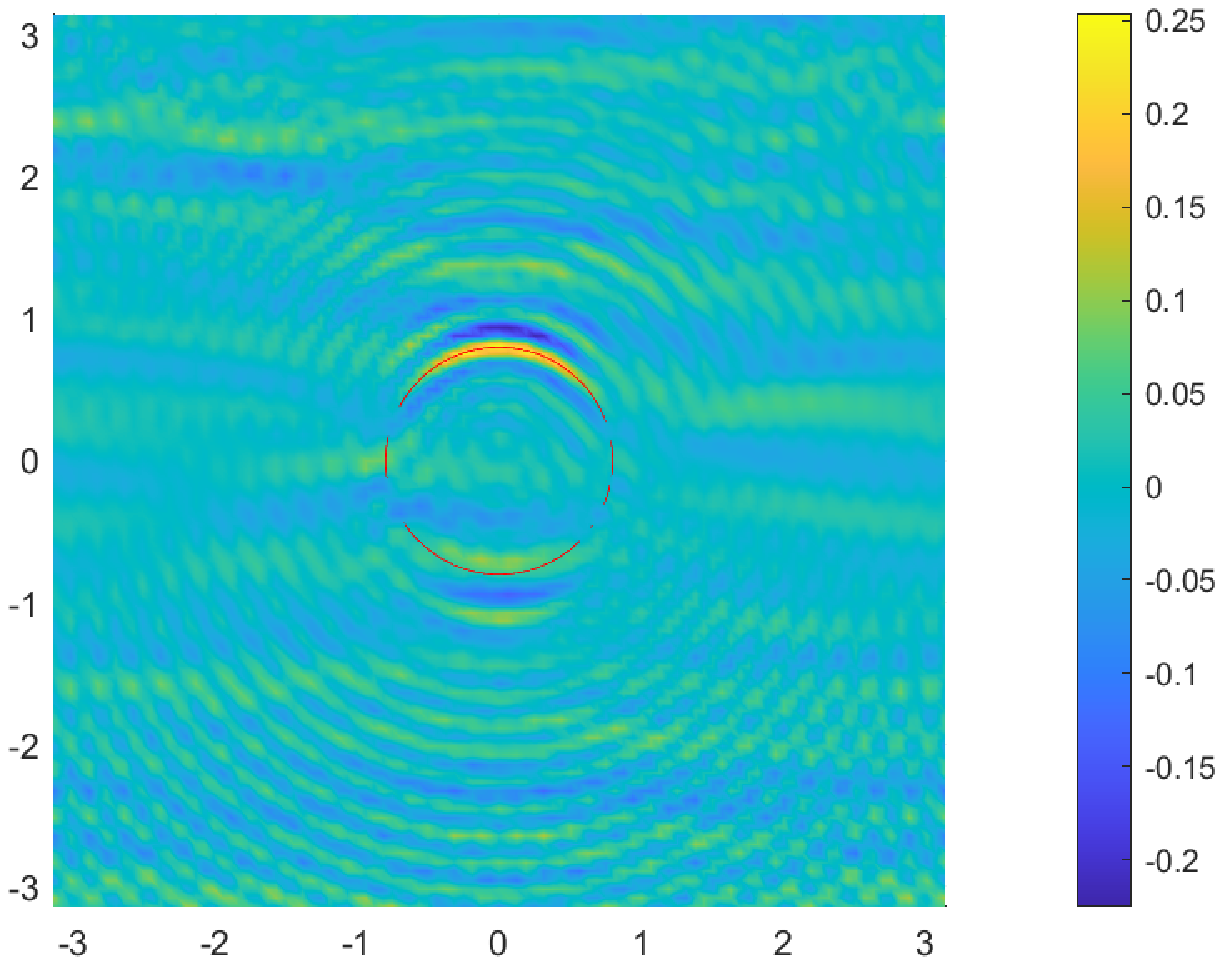}}
\caption{Reconstruction by upper RTM for penetrable circle.}\label{fig4}
\end{figure}
\paragraph{Example 2}
In this example, we consider the imaging of sound-soft periodic kite arrays. Here, $\rho=0.6$, and $k = 4.68\pi$. Figure \ref{fig5} shows the imaging results of $\mathcal{I}_{L}$, which has positive values and captures the non-convexity of the vertical part clearly, and with enough $\alpha$s, even the upper part of the obstacle array is obtained. This confirms our resolution analysis for lower RTM method \eqref{SRTML1}.
 Figure \ref{fig6} shows the imaging results of $\mathcal{I}_{U}$. We can find that the upper horizontal part of the sound-soft periodic kite can be reconstructed clearly. With more $\alpha$s, the imaging quality is also sharper with fewer false images.
 \begin{figure}[ht] 
\subfloat[\label{fig5:a}one $\alpha$]{\includegraphics[width=5cm]{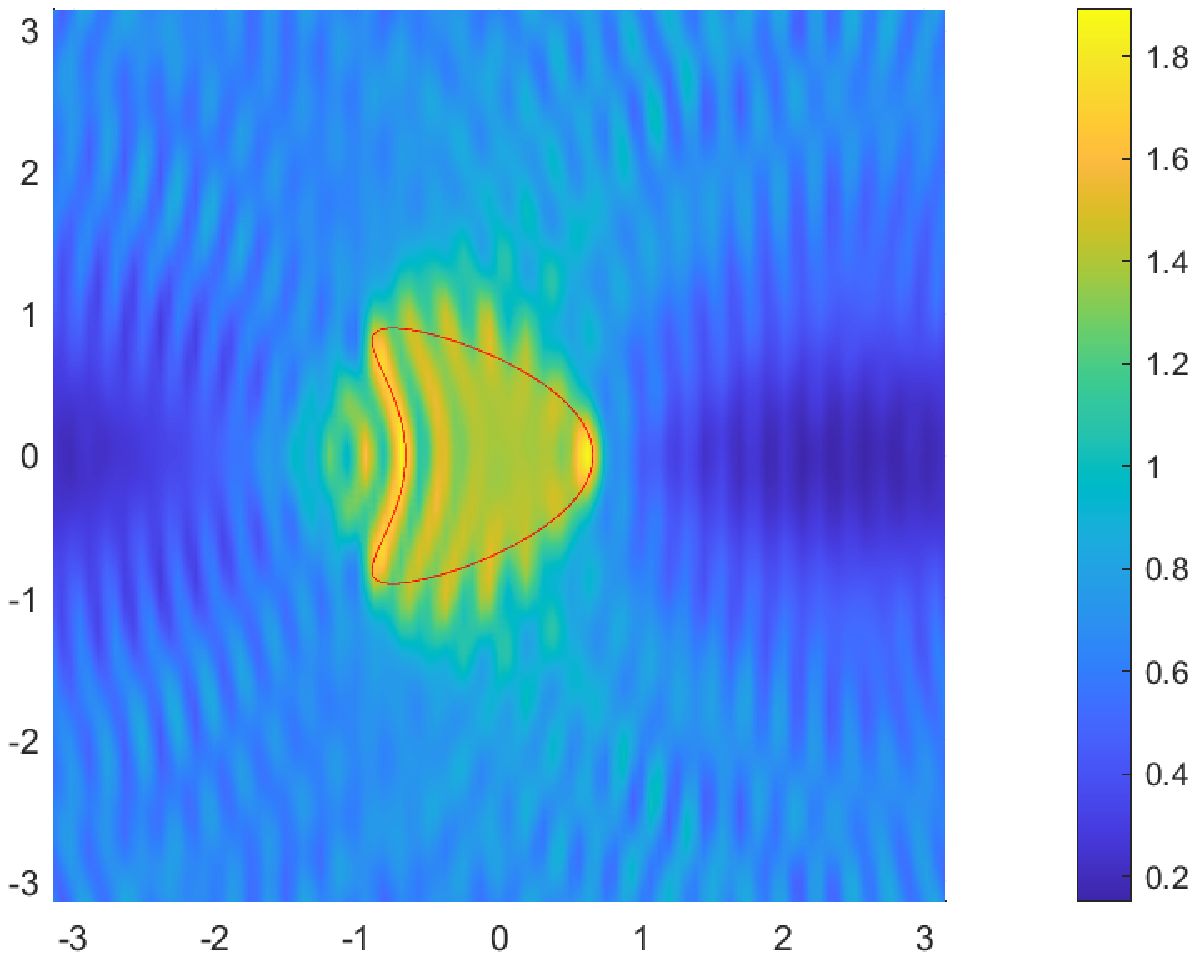}}
\subfloat[\label{fig5:b}five $\alpha$s]{\includegraphics[width=5cm]{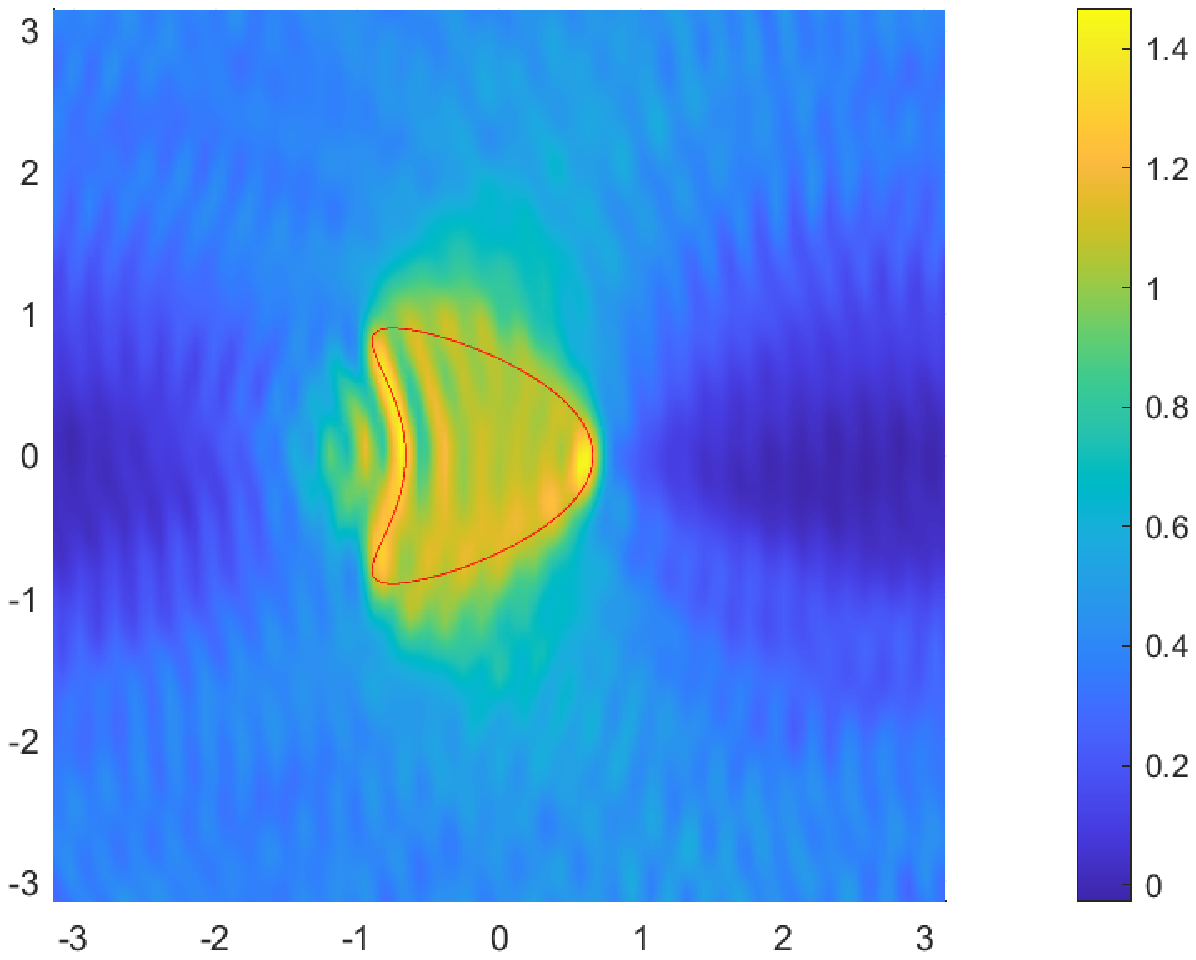}}
\subfloat[\label{fig5:c}nine $\alpha$s]{\includegraphics[width=5cm]{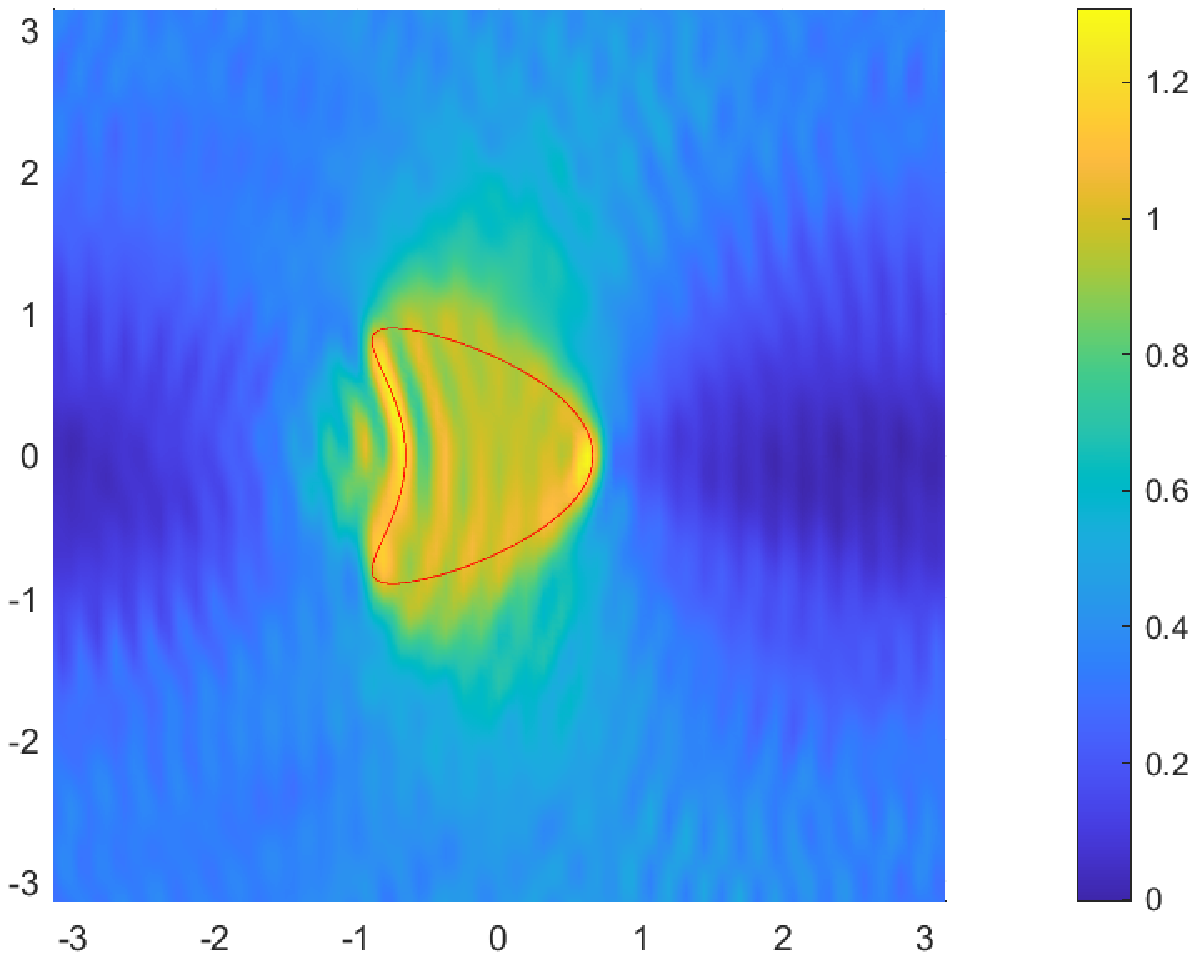}}
\caption{Reconstruction by lower RTM for sound-soft kite.}\label{fig5}
\end{figure}
\begin{figure}
\subfloat[\label{fig6:a}one $\alpha$]{\includegraphics[width=5cm]{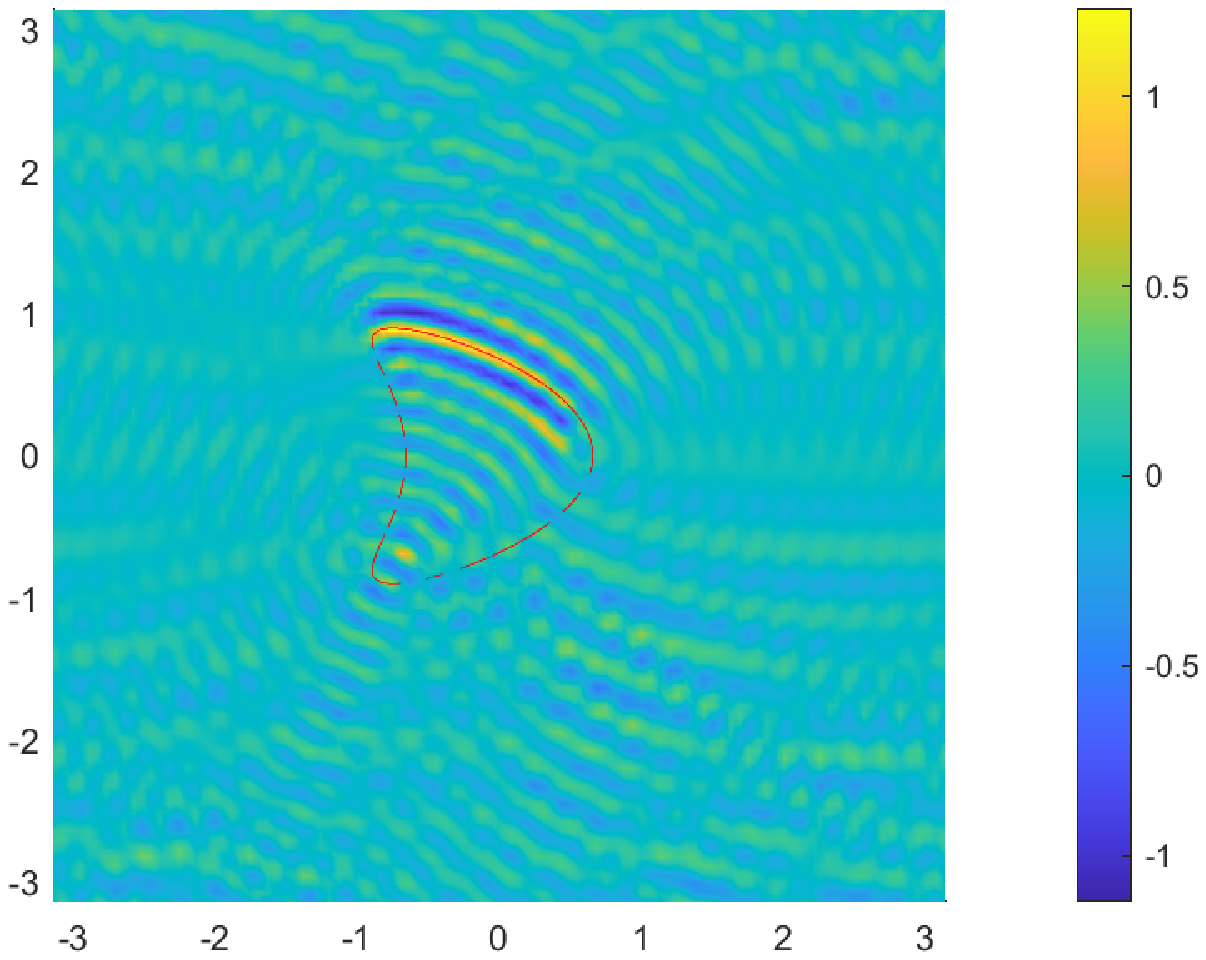}}
\subfloat[\label{fig6:b}five $\alpha$s]{\includegraphics[width=5cm]{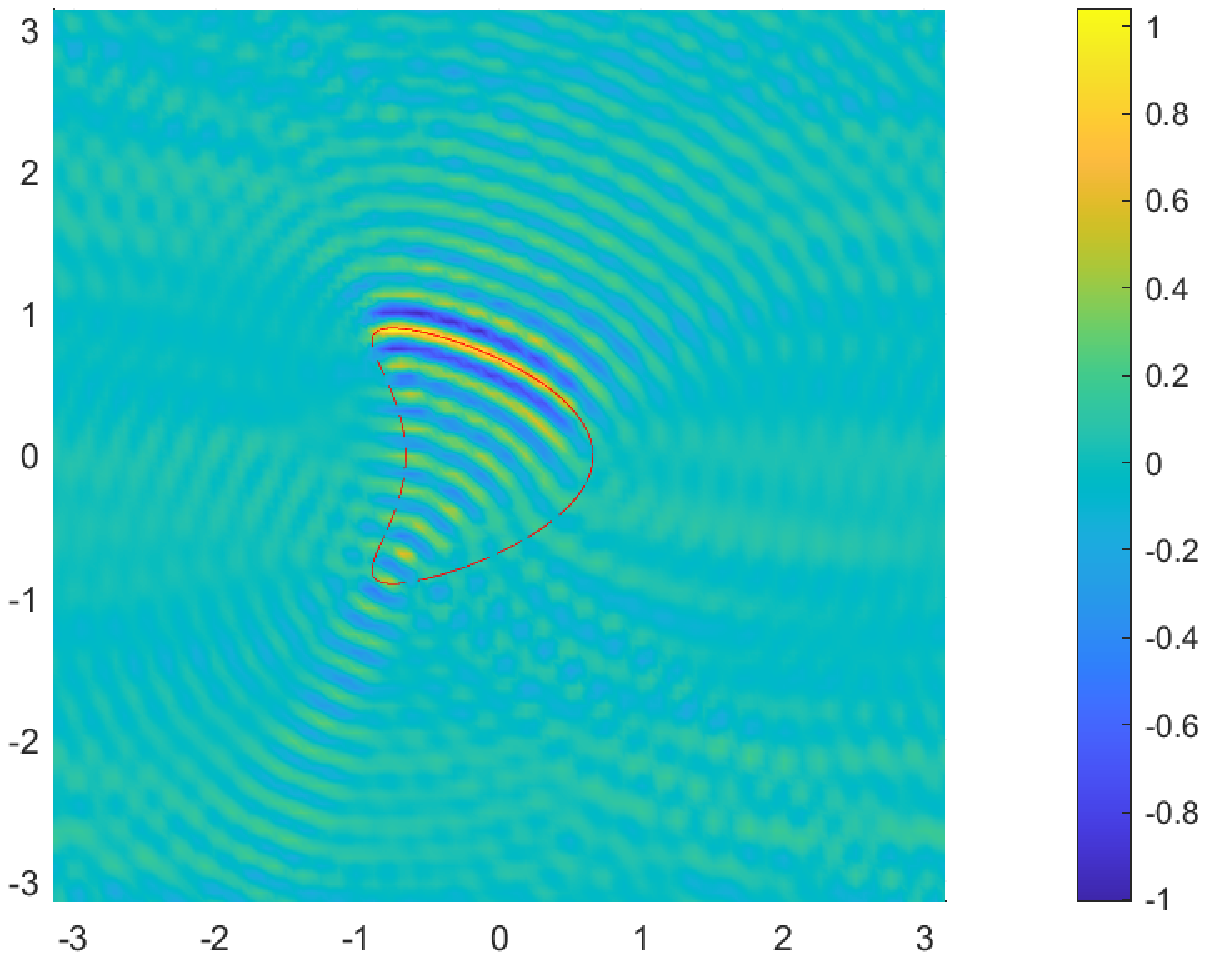}}
\subfloat[\label{fig6:c}nine $\alpha$s]{\includegraphics[width=5cm]{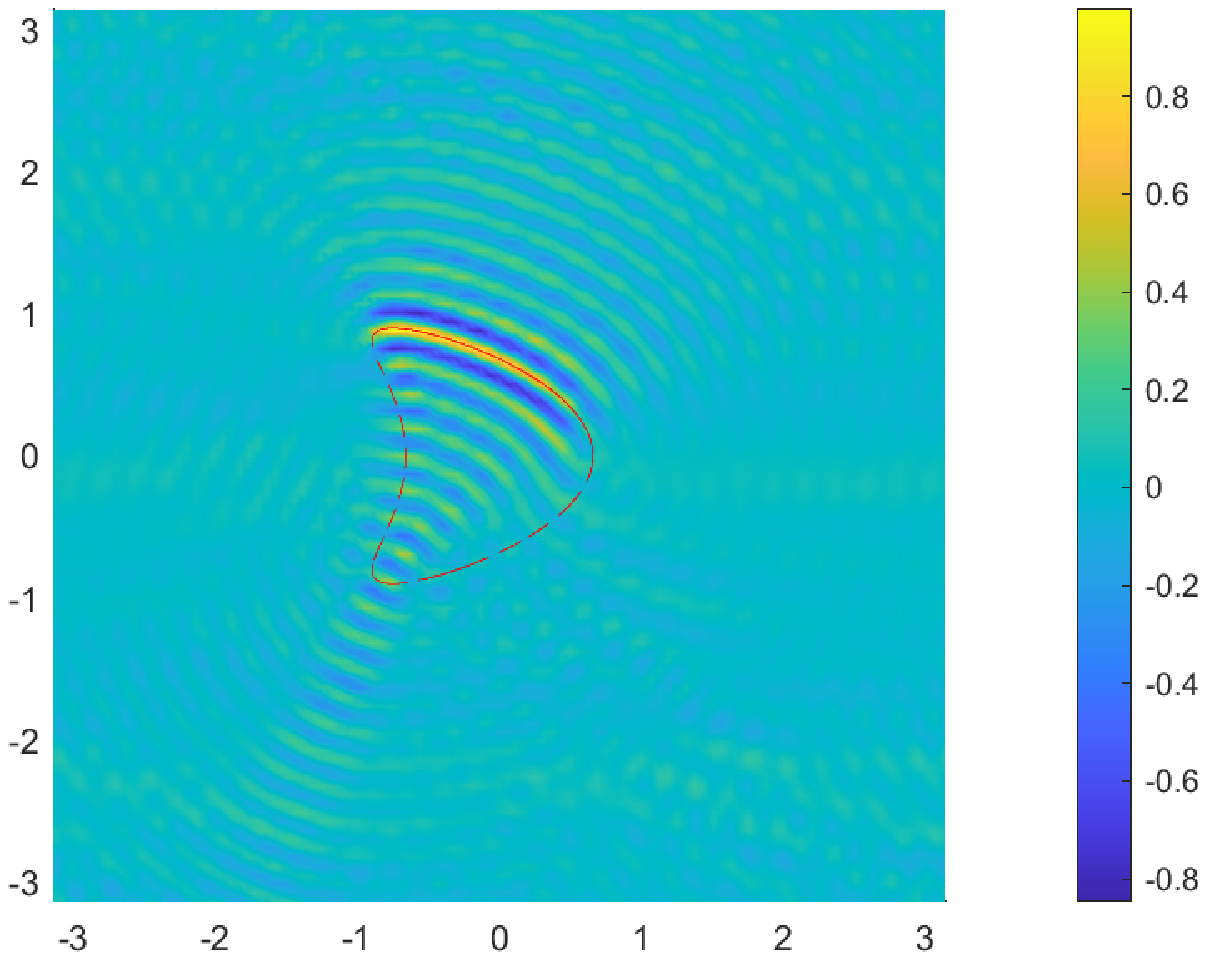}}
\caption{Reconstruction by upper RTM for sound-soft kite.}\label{fig6}
\end{figure}
\paragraph{Example 3}
In this example, we consider the stability of our RTM functionals with respect to the complex additive Gaussian random noise as in \cite{Chen1} on the peanut like scatterer with $\rho=0.2$ at $k = 4.2\pi$. Since there are $N_{r}$ measured data on $\Gamma_{h}$ (or $\Gamma_{-h}$) for any $n\in B_\alpha$,  the received data form an $N_{r}\times |B_{\alpha}|$ matrix for each $\alpha$, we name it $U^{s}_{\alpha}$, thus we introduce
the additive Gaussian noise as follows,
\begin{eqnarray*}
U^{s}_{\alpha,noise}=U^{s}_{\alpha}+V^{noise}_{\alpha},
\end{eqnarray*}
$V^{noise}_{\alpha}$ is the gaussian noise of mean zero with standard deviation of $\mu$ multiplied by the maximum of the data $|U^{s}_{\alpha}|$
\begin{eqnarray*}
U^{noise}_{\alpha}=\frac{\mu \cdot max|U^{s}_{\alpha}|}{\sqrt{2}}(\epsilon_{1}+i\epsilon_{2}).
\end{eqnarray*}
Here,
$\epsilon_{j}\sim\mathcal{N}(0,1)$ for the real ($j = 1$) and imaginary part
($j = 2$)). The noise level is calculated as $|V^{noise}_{\alpha}|^{2}_{l^{2}}=\frac{1}{N_{r}|B_{\alpha}|}\sum |V^{noise}_{\alpha_{n}}(x_{r})|^{2}$, and $\sigma=\mu\cdot max|U^{s}_{\alpha}|$, while the received data level is calculated as $|U^{s}_{\alpha}|^{2}_{l^{2}}=\frac{1}{N_{r}|B_{\alpha}|}\sum |U^{s}_{\alpha_{n}}(x_{r})|^{2}$ for each $\alpha$, and is taken arithmetic mean over all 9 $\alpha$s. The result are listed below in Table \ref{table1} and Table \ref{table2}.

\begin{table}[htbp]
    \centering
    \begin{tabular}{c|c|c|c}
    \hline\hline\noalign{\smallskip}
    $\mu$ &  $\sigma$ & $|U^{s}_{\alpha}|_{l^{2}}$
    &$|V^{noise}_{\alpha}|_{l^2}$\\
    \noalign{\smallskip}\hline\noalign{\smallskip}
0.100000  &  0.255456 &  0.347392  &  0.084532  \\
0.200000  &  0.510912 &  0.347392  &  0.170810  \\
0.400000  &  1.021823 &  0.347392  &  0.341842  \\
0.600000  &  1.532735 &  0.347392  &  0.508574  \\
\noalign{\smallskip}\hline
    \end{tabular}
    \caption{Different levels of average signal and noises for lower RTM averaging over 9 different $\alpha$s.}
    \label{table1}
\end{table}
\begin{figure}[ht] 
\subfloat[\label{fig7:a}10 $\%$ noise]{\includegraphics[width=4cm]{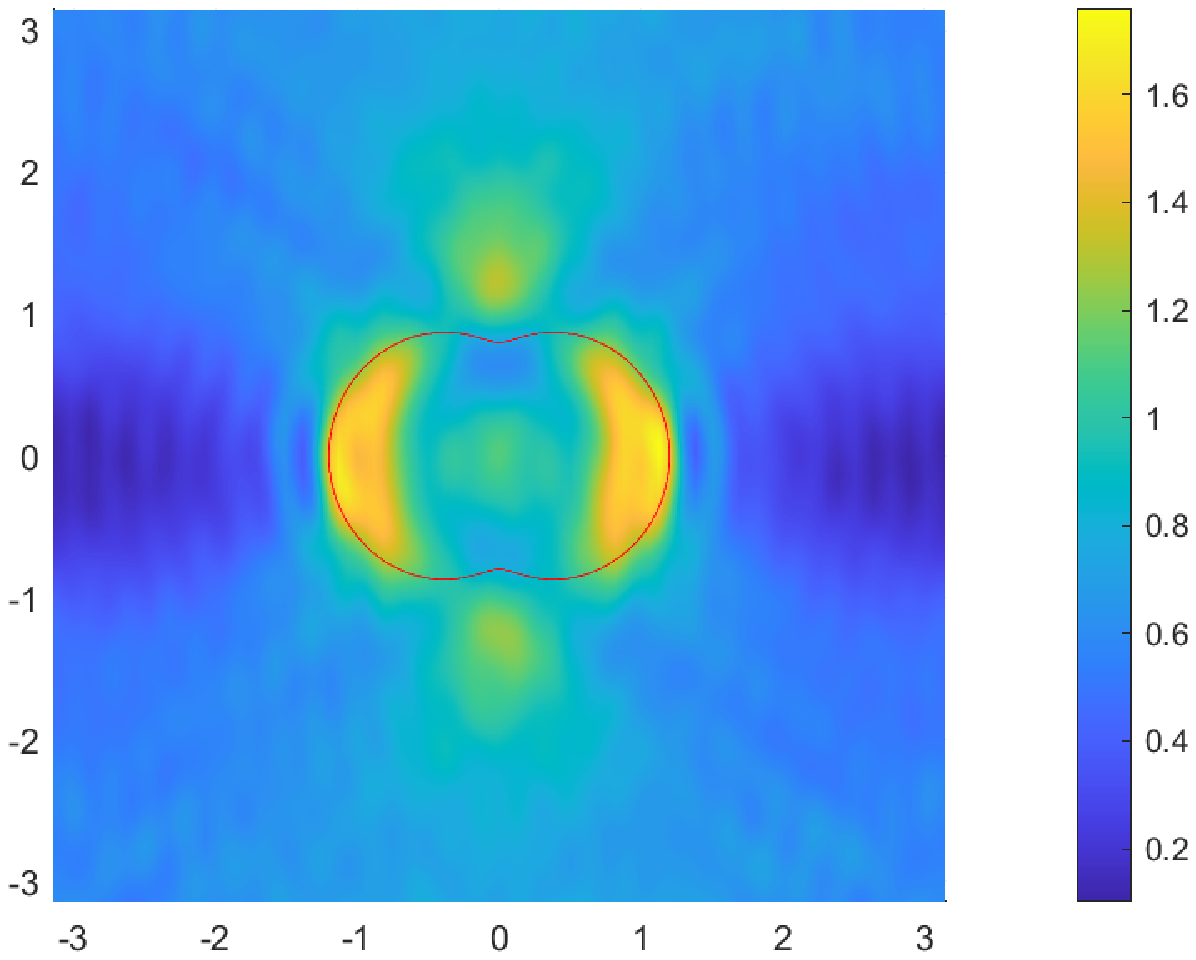}}
\subfloat[\label{fig7:b}20 $\%$ noise]{\includegraphics[width=4cm]{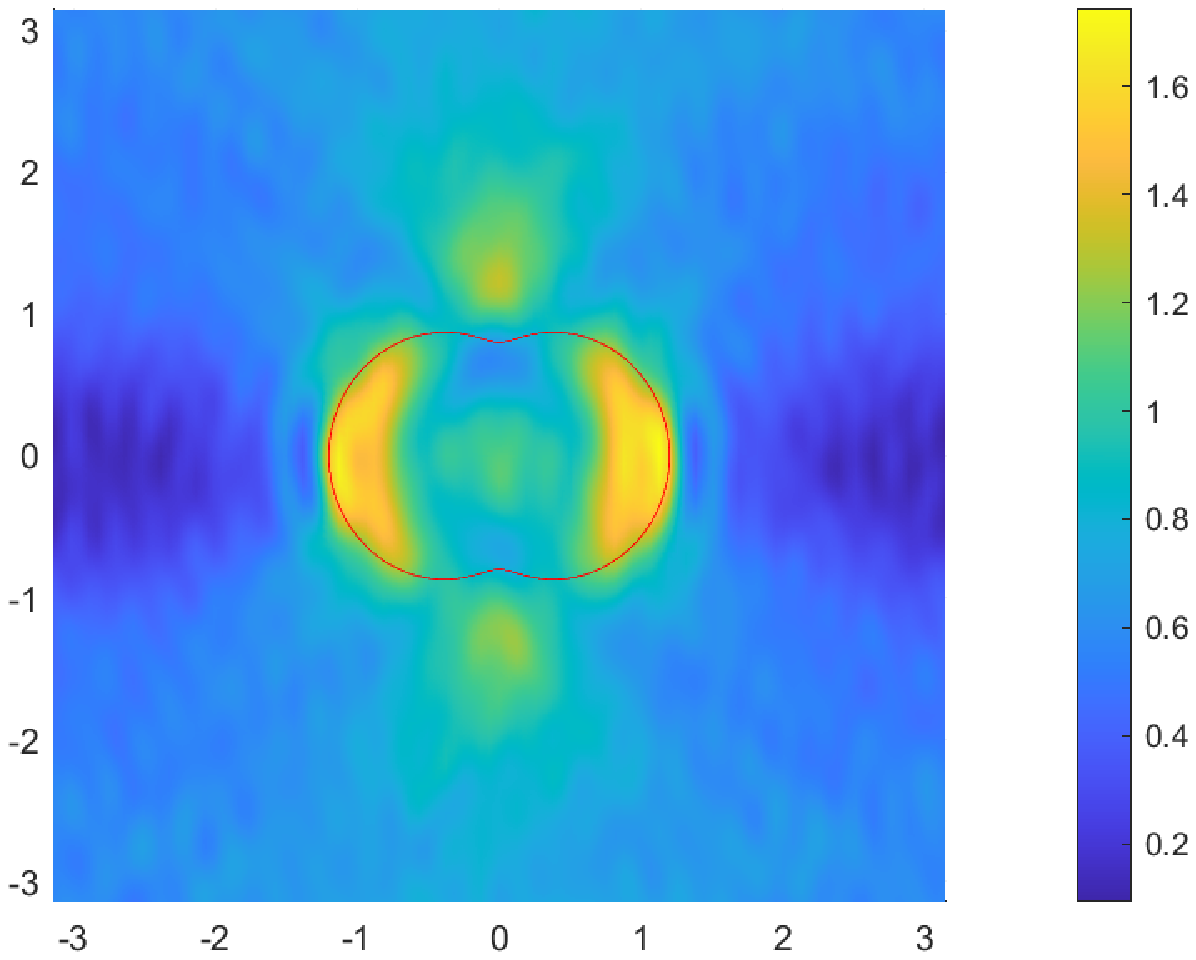}}
\subfloat[\label{fig7:c}40 $\%$ noise]{\includegraphics[width=4cm]{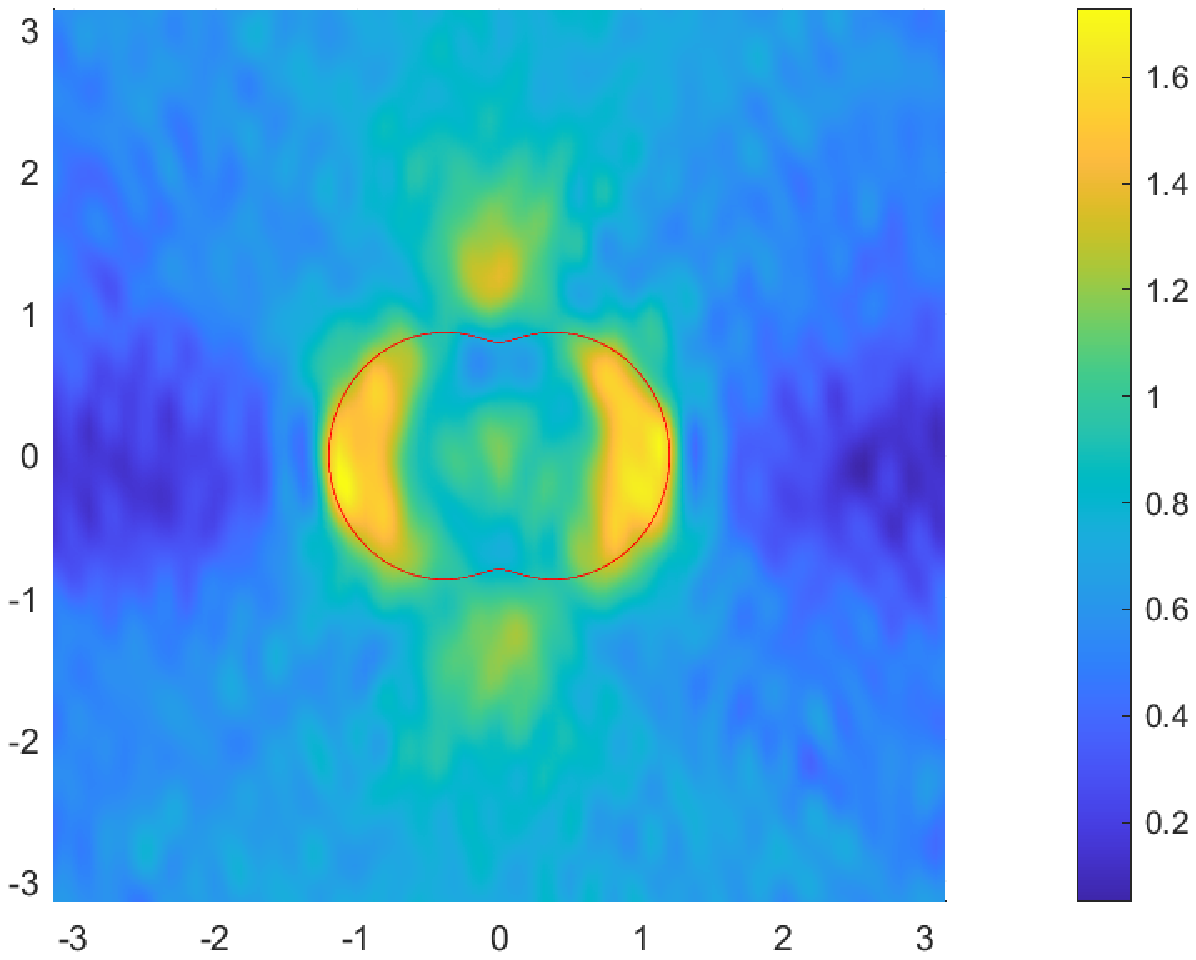}}
\subfloat[\label{fig7:d}60 $\%$ noise]{\includegraphics[width=4cm]{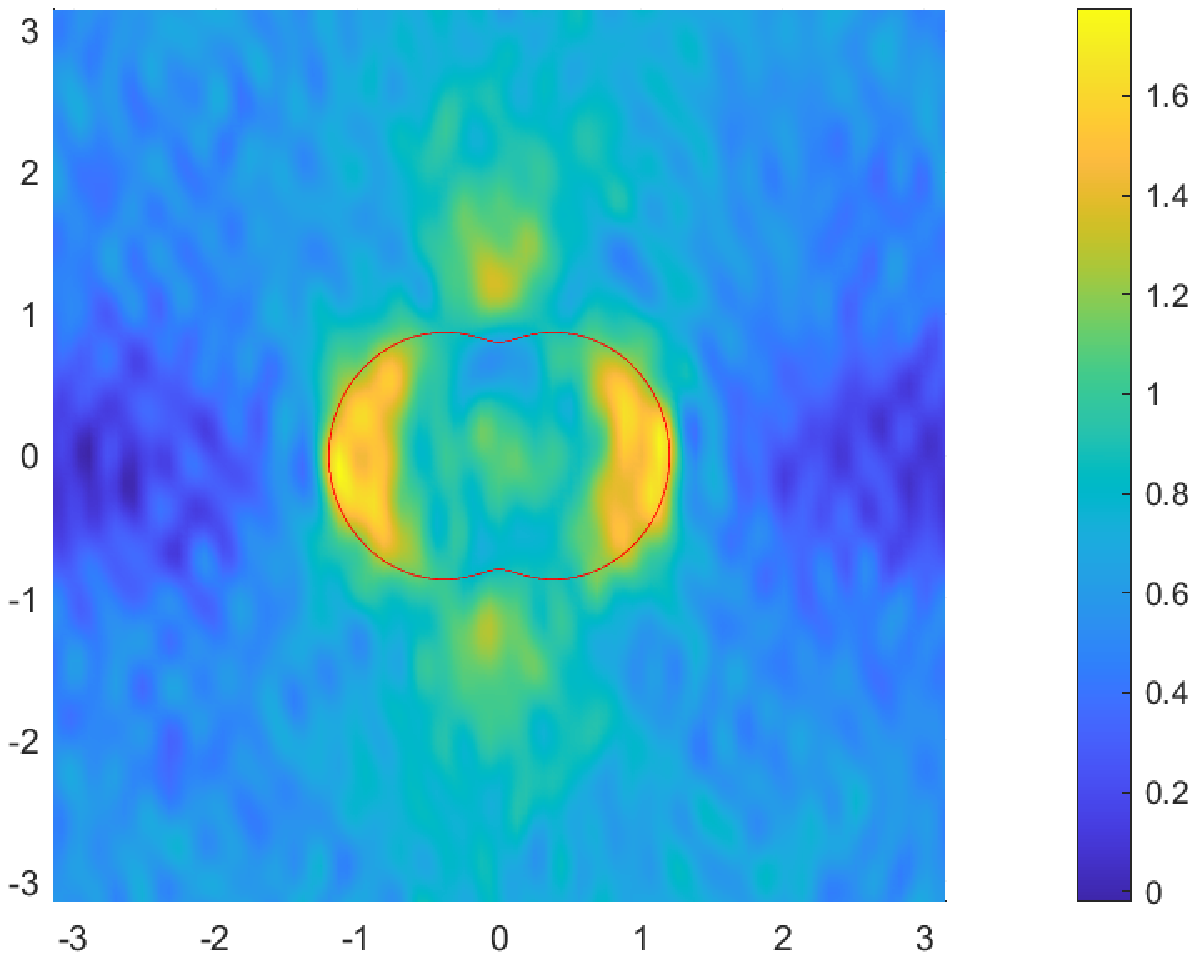}}
\caption{Reconstruction by lower RTM for penetrable peanut with noise level 10$\%$, 20$\%$, 40$\%$, 60$\%$.}\label{fig7}
\end{figure}
\begin{table}[htbp]
    \centering
    \begin{tabular}{c|c|c|c}
    \hline\hline\noalign{\smallskip}
    $\mu$ &  $\sigma$ & $|U^{s}_{\alpha}|_{l^{2}}$
    &$|V^{noise}_{\alpha}|_{l^2}$\\
    \noalign{\smallskip}\hline\noalign{\smallskip}
0.100000  &  0.139696 &  0.149150 &  0.047054 \\
0.200000  &  0.279392 &  0.149150 &  0.094292 \\
0.400000  &  0.558783 &  0.149150 &  0.188035 \\
0.600000  &  0.838175 &  0.149150 &  0.282063 \\
\noalign{\smallskip}\hline
    \end{tabular}
    \caption{Different levels of average signal and noises for upper RTM averaging over 9 different $\alpha$s.}
    \label{table2}
\end{table}
Here we use images of 9 different $\alpha$s, and (a) is the image of noise level of $10\%$, while (b)-(d)correspond to noise level of  $20\%$  to $60\%$. Here Figure \ref{fig7}  shows the imaging quality of $\mathcal{I}_{L}$ of the vertical part of penetrable periodic peanut. Figure \ref{fig8} shows the imaging quality of $\mathcal{I}_{U}$ of the horizontal part of penetrable peanut. The experiments demonstrate that even with large amount of additive noise in the received data, the imaging functional still give the image of the boundary of the obstacle arrays.
\begin{figure} 
\subfloat[\label{fig8:a}10 $\%$ noise]{\includegraphics[width=4cm]{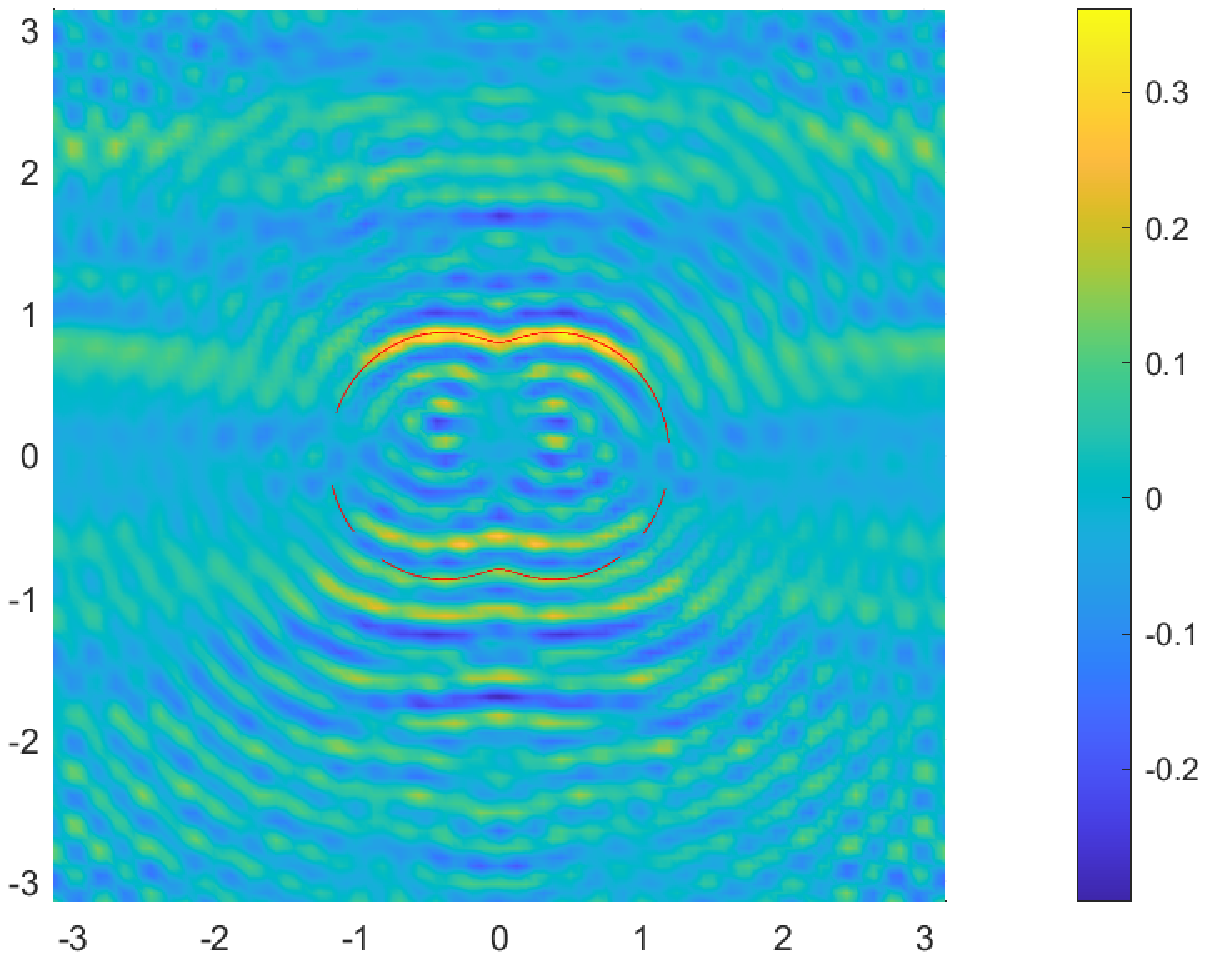}}
\subfloat[\label{fig8:b}20 $\%$ noise]{\includegraphics[width=4cm]{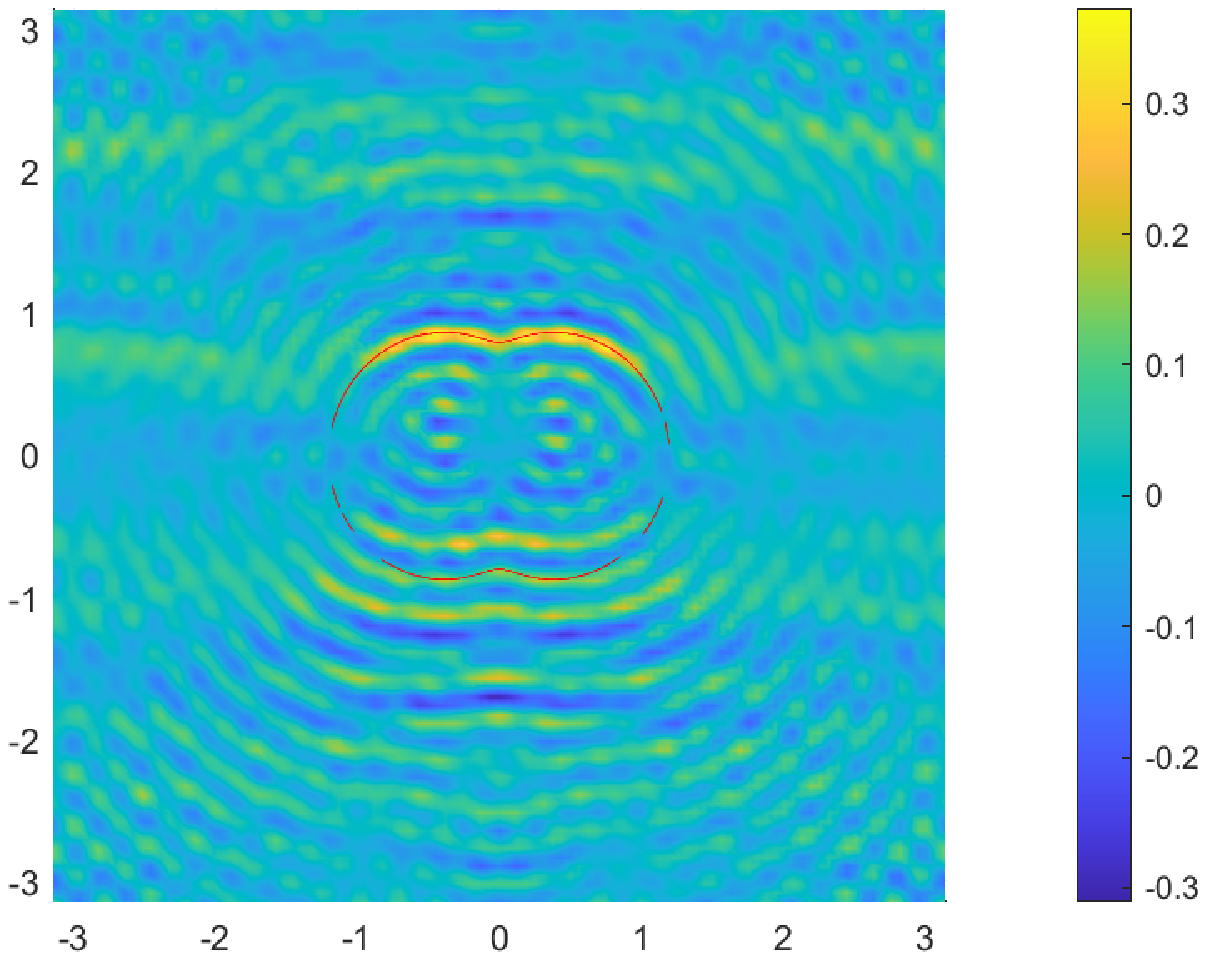}}
\subfloat[\label{fig8:c}40 $\%$ noise]{\includegraphics[width=4cm]{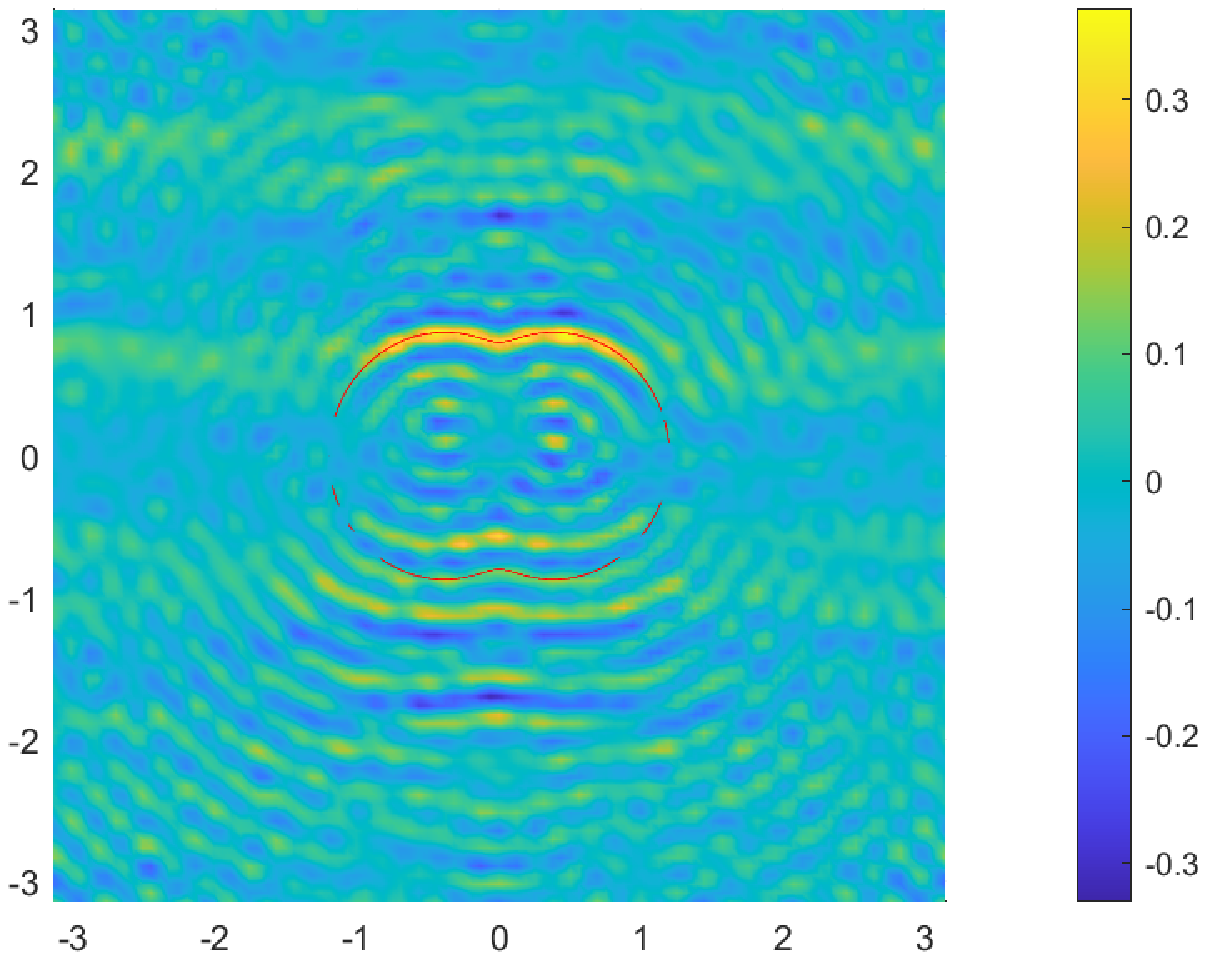}}
\subfloat[\label{fig8:d}60 $\%$ noise]{\includegraphics[width=4cm]{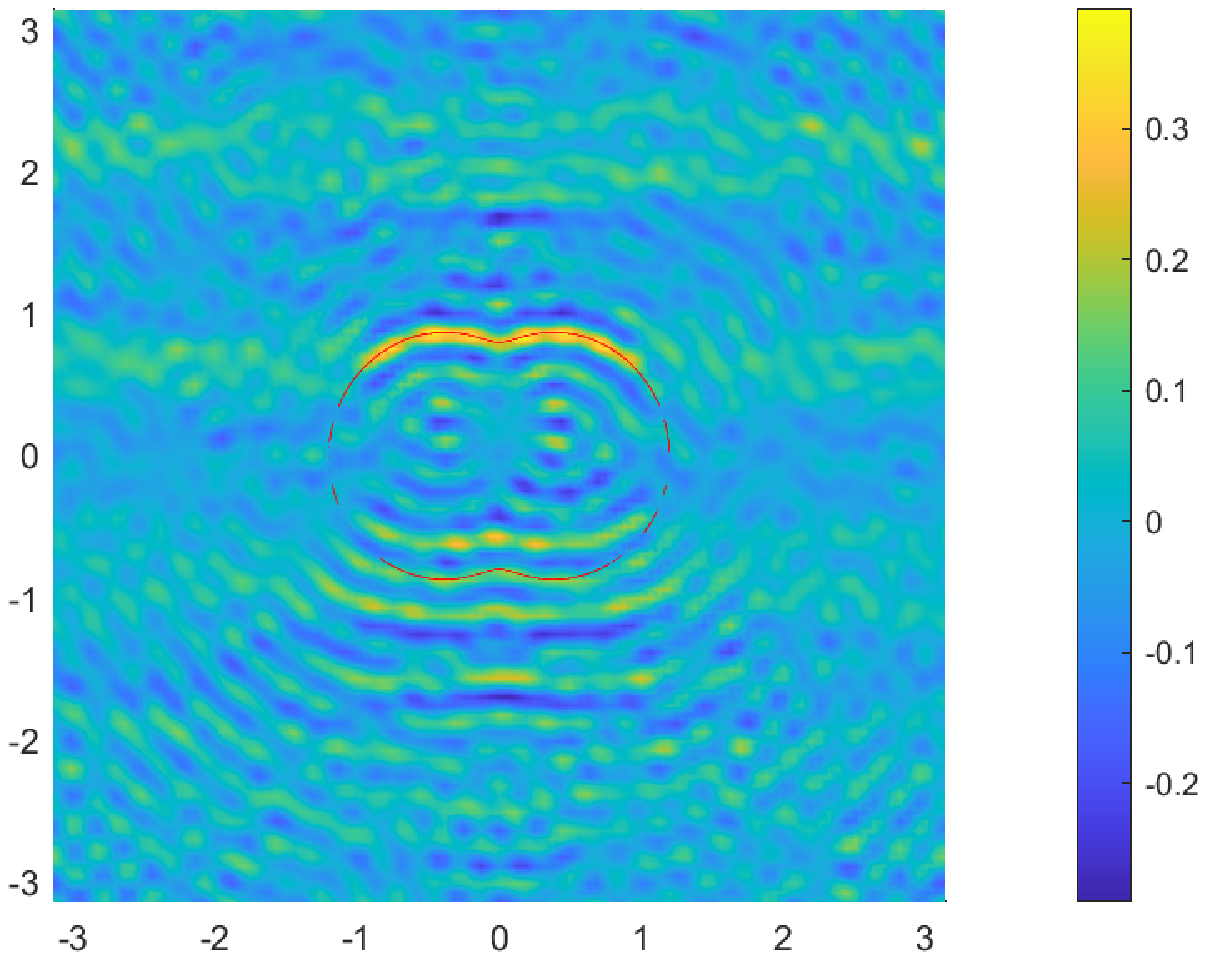}}
\caption{Reconstruction by upper RTM for penetrable peanut with noise level 10$\%$, 20$\%$, 40$\%$, 60$\%$.}\label{fig8}
\end{figure}

\end{document}